%% file: Analysis.tex
\documentclass{amsart}
\usepackage[margin=1in]{geometry}
\usepackage{amsmath, amsfonts, amssymb}
\usepackage{color}
\usepackage[shortlabels]{enumitem}
\usepackage{caption}
\usepackage{overpic}

\newcommand\bx{\boldsymbol{x}}
\newcommand\be{\boldsymbol{e}}

\DeclareMathOperator\tr{tr}
\DeclareMathOperator\re{Re}
\DeclareMathOperator\im{Im}
\DeclareMathOperator\diag{diag}
\DeclareMathOperator*\argmin{argmin}
\DeclareMathOperator\arctanh{arctanh}
\newtheorem{theorem}{Theorem}
\newtheorem{lemma}{Lemma}

\title{How does Gauge Cooling Stabilize Complex Langevin?}
\author{Zhenning Cai}
\address[Zhenning Cai]{Department of Mathematics, National University of Singapore,
  Level 4, Block S17, 10 Lower Kent Ridge Road, Singapore 119076}
\email{matcz@nus.edu.sg}
\date{\today}

\author{Yana Di}
\address[Yana Di]{LSEC, Academy of Mathematics and Systems Science, Chinese Academy of Sciences, Beijing 100190, China, and 
	University of Chinese Academy of Sciences, Beijing 100049, China}
\email{yndi@lsec.cc.ac.cn}

\author{Xiaoyu Dong}
\address[Xiaoyu Dong]{LSEC, Academy of Mathematics and Systems Science, Chinese Academy of Sciences, Beijing 100190, China}
\email{dongxiaoyu@lsec.cc.ac.cn}

\thanks{Zhenning Cai's work was supported by National University of Singapore Startup Fund under grant No. R-146-000-241-133.
Yana Di's work was supported by National Natural Science Foundation of China Nos. 91630208, 91641107, and 11771437. The authors would like to thank Dr. Lei Zhang in Department of Mathematics, National University of Singapore for useful discussions.}

\begin{document}

\input{Abstract.tex}

\maketitle

\input{Introduction.tex}
\input{ComplexLangevin.tex}

\input{GaugeCooling.tex}
\input{Eigenvalue.tex}
\input{Examples.tex}
\input{Conclusion.tex}

\bibliographystyle{amsplain}
\bibliography{Analysis}
\end{document}

%% file: Abstract.tex
\begin{abstract}
We study the mechanism of the gauge cooling technique to stabilize the complex Langevin method in the one-dimensional periodic setting. In this case, we find the exact solutions for the gauge transform which minimizes the Frobenius norm of link variables. Thereby, we derive the underlying stochastic differential equations by continuing the numerical method with gauge cooling, and thus provide a number of insights on the effects of gauge cooling. A specific case study is carried out for the Polyakov loop model in $SU(2)$ theory, in which we show that the gauge cooling may help form a localized distribution to guarantee there is no excursion too far away from the real axis.
\end{abstract}

%% file: Introduction.tex
\section{Introduction}
In quantum chromodynamimcs (QCD), the renormalization of the coupling constant depends on the energy scale. As the energy scale increases, the coupling constant decays to zero. Therefore the perturbative theory works well for high-energy scattering. However, when studying QCD at small momenta or energies (less than 1GeV), the coupling constant is comparable to 1 and the perturbative theory is no longer accurate \cite{Greiner2007}. In this case, one of the important methods is the path integral formulation, in which people usually employ the lattice gauge theory to perform calculations. In lattice QCD, the degrees of freedom for both gluons and quarks are discretized on a four-dimensional lattice, whose grid points are
\begin{displaymath}
\bx = (i,j,k,t)a, \qquad i,j,k,t \in \mathbb{Z},
\end{displaymath}
where $a$ is the size of the lattice. Gluons on the lattice are represented by link variables between lattice points, which are matrices $U_{\mu}(\bx) \in SU(3)$, denoting the link between lattice points $\bx$ and $\bx + \be_{\mu}$. Since the degrees of freedom for quarks can usually be integrated out explicitly, the final form of the path integral is given by the following partition function:
\begin{equation} \label{eq:Z}
Z = \int [\mathrm{d}U] \det M(\{U\}) \, \mathrm{e}^{-S(\{U\})},
\end{equation}
where $\int [\mathrm{d}U]$ stands for the integral with respect to all link variables $U_{\mu}(\bx)$ defined on the Haar measure of $SU(3)$, and $\{U\}$ represents the collection of all link variables. In the integrand, the matrix $M(\{U\})$ is the fermion Green's function, and $S(\{U\})$ is the Euclidean action for the gluons. Thus, given an observable $O(\{U\})$, its expected value can be calculated by
\begin{displaymath}
\langle O \rangle = \frac{1}{Z} \int [\mathrm{d}U] O(\{U\}) \det M(\{U\}) \, \mathrm{e}^{-S(\{U\})}.
\end{displaymath}
Monte Carlo methods such as the Metropolis algorithm and the Langevin algorithm can be applied to evaluate this integral.

When we consider the system with quark chemical potential, the term $\det M(\{U\})$ may be non-positive, and thus \eqref{eq:Z} is not a valid partition function \cite{Gattringer2010}. In such a circumstance, the ``reweighting'' technique is required to carry out the Monte Carlo simulation. Such a method introduces another partition function
\begin{displaymath}
Z_0 = \int [\mathrm{d}U] \det M_0(\{U\}) \, \mathrm{e}^{-S(\{U\})}
\end{displaymath}
and rewrite $Z$ as
\begin{displaymath}
Z = Z_0 \left\langle \frac{\det M(\{U\})}{\det M_0(\{U\})} \right\rangle_0,
\end{displaymath}
where $\langle \cdot \rangle_0$ denotes the expectation of $\cdot$ based on the partition function $Z_0$. However, due to the rapid change of sign in $\det M(\{U\}) / \det M_0(\{U\})$, significant numerical sign problem may appear, causing large deviation in the numerical integration \cite{Forcrand2009}.

To relax the sign problem, numerical methods such as Lefschetz thimble method \cite{Cristoforetti2012} and complex Langevin method (CLM) \cite{Parisi1983} have been introduced. This paper focuses on the CLM, which can be considered as a straightforward complexification of the real Langevin method. While the complex Langevin method effectively relaxes the sign problem for some problems, the behavior of this method seems quite unpredictable. Although in a lot of cases, the method produces correct integral values, sometimes it provides incorrect integral values, or even generates diveregent dynamics \cite{Ambjorn1986}. A number of efforts have been made to figure out the reason of failure and find a theory for its correct convergence \cite{Gausterer1994, Gausterer1998, Aarts2010, Nagata2016, Scherzer2019}. Although a complete theory has not been found, the problem has been much better understood.

On the other hand, people are trying to stabilize CLM by various numerical techniques such as dynamical stabilization \cite{Attanasio2018} and gauge cooling \cite{Seiler2013}. The dynamical stabilization adds a regularization term to the dynamics, which avoids divergence by introducing small biases; the gauge cooling technique makes use of the gauge symmetry of the system to stabilize CLM, and therefore does not introduce any biases. The gauge cooling technique has been successfully applied to a number of problems \cite{Sexty2014, Nagata2016G, Aarts2016}, which shows a considerable enhancement of the stability of CLM. Compared with a large number of studies for the original CLM, the analysis for the gauge cooling technique is still rarely seen. The justification of such a technique has been studied in \cite{Nagata2016, Nagata2016J}, while they do not explain how the complex Langevin process is stabilized. In this paper, we would like to carry out some numerical analysis for CLM with gauge cooling based on the one-dimensional $SU(n)$ theory. By such analysis, we would like to make the stabilizing effect more explicit. A specific analysis will be carried out for the Polyakov loop model. This model describes an infinitely heavy quark propagating only in time direction. Thus all the link variables can be labelled by a single subscript: $U_k$, $k = 1,\cdots,N$, and we have periodic boundary conditions. Here we follow the reference \cite{Seiler2013} and choose the partition function to be
\begin{equation} \label{eq:ZPolyakov}
Z = \int [\mathrm{d}U] \exp \left( \tr(\beta_1 U_1 \cdots U_N + \beta_2 U_N^{-1} \cdots U_1^{-1}) \right),
\end{equation}
where $\beta_1$ and $\beta_2$ are constants which are allowed to be complex. We are going to reveal some insights for the gauge cooling technique for this 1D model. Especially, when $n = 2$, we will show that the complex Langevin process is localized when $\beta_1 + \beta_2$ is not too large.

The rest of this paper is organized as follows. In Section \ref{sec:CL}, we briefly review the complex Langevin method and the gauge cooling technique. In Section \ref{sec:GC}, the optimal gauge transform is solved for one-dimensional periodic models. Section \ref{sec:eig} is devoted to the study of gauge cooling after continuation of the numerical method, and some numerical experiments are carried out in Section \ref{sec:examples}. Finally, the paper ends with some concluding remarks in Section \ref{sec:conclusion}.

%% file: ComplexLangevin.tex
\section{Complex Langevin method and gauge cooling} \label{sec:CL}
In this section, we provide a brief review of the CLM and the gauge cooling technique. For the sake of simplicity, we only introduce these methods for the one-dimensional lattice gauge theory,

\subsection{Complex Langevin method}
For the Polyakov chain model with partition function \eqref{eq:ZPolyakov}, the expected value of the observable is
\begin{equation} \label{eq:observable}
\langle O \rangle = \frac{1}{Z} \int [\mathrm{d}U] O(\{U\}) \, \mathrm{e}^{-S(\{U\})},
\end{equation}
where $\{U\} = (U_1, U_2, \cdots, U_N) \in [SU(n)]^N$ with $N$ being the number of lattice points. Here the Euclidean action $S(\{U\})$ is defined by
\begin{equation} \label{eq:action}
S(\{U\}) = -\tr(\beta_1 U_1 \cdots U_N + \beta_2 U_N^{-1} \cdots U_1^{-1}).
\end{equation}
Since $SU(n)$ has $n^2-1$ infinitesimal generators, the above integral is an $(n^2-1)N$-dimensional integral, which needs to be done by Monte Carlo method. When $\beta_1$ and $\beta_2$ are real and equal, we have
\begin{displaymath}
S(\{U\}) = -2 \beta_1 \re \tr(U_1 \cdots U_N) \in \mathbb{R},
\end{displaymath}
which means that
\begin{equation} \label{eq:prob}
P(\{U\}) = \frac{1}{Z} \exp(-S(\{U\}))
\end{equation}
is a probability density function. In this case, to evaluate \eqref{eq:observable}, we just need to draw $K$ samples $\{U^{(k)}\}$, $k=1,\cdots,K$ from the above probability density function, and then compute the average of the observable by
\begin{equation} \label{eq:MC}
\langle O \rangle \approx \frac{1}{K} \sum_{k=1}^K O(\{U^{(k)}\}).
\end{equation}
Thus it remains only to provide a practical algorithm to draw the samples $\{U^{(k)}\}$.

A classical method to draw samples from a partition function is the Langevin method (see e.g. \cite[Section 7.1.10]{Greiner2007}), which turns the sampling problem to solving a stochastic ordinary differential system, called Langevin dynamics, whose invariant measure is the desired probability density function. In our case, the Langevin dynamics has been given in \cite{Gausterer1988}, which reads
\begin{equation} \label{eq:SDE}
\mathrm{d} U_k = - \sum_{a=1}^{n^2-1} \mathrm{i} \lambda_a [U_k D_{ak} S(\{U\}) \,\mathrm{d}t + U_k \circ \mathrm{d} w_{ak}], \qquad k=1,\cdots,N.
\end{equation}
Here $\{w_{ak}\}$ denotes the $N(n^2-1)$-dimensional Brownian motion, and the symbol ``${} \!\circ \mathrm{d} w_{ak}$'' stands for the Stratonovich interpretation of the stochastic differential equation (SDE). The matrices $\lambda_a$ denotes the infinitesimal generators of $SU(n)$ normalized by
\begin{equation} \label{eq:orth}
\tr(\lambda_a \lambda_b) = 2\delta_{ab}, \qquad a,b=1,\cdots,n^2-1.
\end{equation}
When $n=2$, they are usually chosen as Pauli matrices; when $n=3$, Gell-Mann matrices are a common choice. In the drift term, $D_{ak}$ is the left Lie derivative operator defined by
\begin{equation} \label{eq:Lie_derivative}
D_{ak} f(\{U\}) = \lim_{\epsilon \rightarrow 0} \frac{f(\{U^{\epsilon}\}) - f(\{U\})} {\epsilon},
\end{equation}
where $\{U^{\epsilon}\} = (U_1^{\epsilon}, \cdots, U_N^{\epsilon})$ and $U_k^{\epsilon} = \exp(\mathrm{i} \epsilon \delta_{kl} \lambda^a U_k)$. Let $p(t,\{U\})$ be the probability density function at time $t$. Then $p(t,\{U\})$ satisfies the following Fokker-Planck equation:
\begin{displaymath}
\frac{\partial p}{\partial t} - D_{ak} [(D_{ak} S) p] = D_{ak}^2 p,
\end{displaymath}
from which one can see that $P(\{U\})$ defined in \eqref{eq:prob} is a stationary solution of the above equation. Thus, when the stochastic process \eqref{eq:SDE} is ergodic, \textit{i.e.}, there exists a unique invariant measure, then the sampling of $\{U\}$ can be implemented as follows:
\begin{enumerate}[1.]
\item Choose an arbitrary initial value $\{U\} = (U_1, \cdots, U_N) \in [SU(n)]^N$ and a proper time step $\Delta t$. Let $t = 0$. Choose $T > 0$ which indicates the time at which the probability density function $p(t,\cdot)$ is sufficiently close to the invariant measure $P(\cdot)$.
\item Evolve the solution by one time step:
  \begin{equation} \label{eq:evolve}
  U_k \leftarrow \exp \left(-\sum_{a=1}^{n^2-1} \mathrm{i} \lambda_a \left( D_{ak}S(\{U\}) \Delta t + \eta_{ak} \sqrt{2\Delta t} \right) \right) U_k, \qquad k = 1,\cdots,N,
  \end{equation}
  where $\eta_{ak}$ obeys the standard normal distribution. Set $t \leftarrow t + \Delta t$.
\item If $t < T$, return to step 2; otherwise, choose $\Delta T > 0$ as the time interval between two adjacent samplings, and let $T \leftarrow T + \Delta T$.
\item Evolve the solution by one time step according to \eqref{eq:evolve}.
\item If $t > T$, take the current $\{U\}$ as a new sample and set $T \leftarrow T + \Delta T$. Stop if the number of samples is sufficient.
\item Return to step 4.
\end{enumerate}
Here the equation \eqref{eq:evolve} is the forward Euler discretization of the Langevin equation \eqref{eq:SDE}. Usually, when implementing the above algorithm, $O(\{U\})$ is immediately evaluated and added to the sum in \eqref{eq:MC}, so that it is unnecessary to record all the samples.

In this paper, we are more interested in the case in which $\beta_1$ and $\beta_2$ are not equal, or even not real. In such a circumstance, the above derivation is no longer valid since $P(\{U\})$ is not a probability function. It has been introduced in the previous section that numerical sign problem may appear if the reweighting technique is used as a remedy. Interestingly, when the action goes complex, the algorithm introduced in the previous paragraph can still be applied if both the action $S(\cdot)$ and the observable $O(\cdot)$ can be holomorphically extended to the following domain of definition:
\begin{displaymath}
[SL(n,\mathbb{C})]^N = \left\{ \{U\} \in (\mathbb{C}^{n\times n})^N \,\Bigg\vert\,
  U_k = \exp \left( \sum_{a=1}^{n^2-1} \mathrm{i} U_{ak} \lambda_a \right), ~~ U_{ak} \in \mathbb{C}, ~~ k=1,\cdots,N \right\}.
\end{displaymath}
Note that on the right-hand side of the above equation, if $U_{ak} \in \mathbb{C}$ is replaced by $U_{ak} \in \mathbb{R}$, then the set is exactly $[SU(n)]^N$. Therefore $[SL(n,\mathbb{C})]^N$ can be considered as a complexification of $[SU(n)]^N$. In this case, the algorithm with exactly the same procedure is called ``complex Langevin method''. It can be formally shown that the complex Langevin method also produces an approximation of $\langle O \rangle$ \cite{Aarts2010, Nagata2016, Nagata2016J, Scherzer2019}. However, it has been demonstrated in \cite{Seiler2013, Seiler2018} that the complex Langevin method does not provide correct results when $N$ is large. This is also confirmed in our numerical experiments presented in Section \ref{sec:examples}.

\subsection{Gauge cooling technique}
It is generally accepted that when the stochastic variable $\{U\}$ drifts too far away from the space $[SU(n)]^N$, the computation tends to provide incorrect results or even fails to converge \cite{Nagata2016, Seiler2018}. In \cite{Seiler2013}, a technique called ``gauge cooling'' was proposed to relax such a problem without introducing any biases to the system. The key idea of gauge cooling is to apply a proper gauge transform after every time step to keep $\{U\}$ close to $[SU(n)]^N$. Such a gauge transform is defined in $[SL(n,\mathbb{C})]^N$, which is called the ``complexified gauge transform''. Due to the gauge symmetry of the lattice gauge theory, such a gauge transform changes neither the action nor the observable. Therefore formally, the process still gives correct expected value. This technique is applicable for the full four-dimensional case. In this paper, we only provide the details for the one-dimensional case with periodic boundary condition.

For the Polyakov chain model, for any field $\{U\} \in [SL(n,\mathbb{C})]^N$, the complexified gauge transform reads
\begin{equation} \label{eq:gc}
\widetilde{U}_k := V_k^{-1} U_k V_{k+1}, \qquad k = 1,\cdots,N,
\end{equation}
where $\{\widetilde{U}\} = (\widetilde{U}_1, \cdots, \widetilde{U}_N)$ is the field after transformation, and $\{V\} = (V_1, \cdots, V_N) \in [SL(n, \mathbb{C})]^N$ defines the gauge transform. Due to the periodic boundary condition, in \eqref{eq:gc}, when $k = N$, the matrix $V_{N+1}$ is identical to $V_1$. The gauge symmetry says that for any $\{U\}, \{V\} \in [SL(n,\mathbb{C})]^N$, we always have
\begin{displaymath}
O(\{U\}) = O(\{\widetilde{U}\}), \qquad S(\{U\}) = S(\{\widetilde{U}\}).
\end{displaymath}
For any $\{U\}$ obtained from equation \eqref{eq:evolve}, our aim is to find $\{V\} \in [SL(n,\mathbb{C})]^N$ such that $\{\widetilde{U}\}$ is close to $[SU(n)]^N$. Different measurement of the closeness has been used in the literature \cite{Seiler2013, Nagata2016}, and here we adopt the one proposed in \cite{Seiler2013}, which is based on the following theorem:
\begin{theorem} \label{thm:F_norm}
Define
\begin{displaymath}
\|\{U\}\| = \left( \sum_{k=1}^N \|U_k\|_F^2 \right)^{1/2}, \qquad \forall \{U\} \in [\mathbb{C}^{n\times n}]^N,
\end{displaymath}
where $\|\cdot\|_F$ is the Frobenius norm of the matrices defined by $\|U_k\|_F^2 = \tr(U_k U_k^{\dagger})$. Then we have
\begin{displaymath}
\|\{U\}\| \geqslant \sqrt{N n}, \qquad \forall \{U\} \in [SL(n,\mathbb{C})]^N,
\end{displaymath}
and the equality holds if and only if $\{U\} \in [SU(n)]^N$.
\end{theorem}

\begin{proof}
For any $k=1,\cdots,N$, assuming $\nu_k^{(1)},\nu_k^{(2)},\cdots,\nu_k^{(n)}$ are all the eigenvalues of $U_k U_k^\dagger$ where $U_k \in SL(n,\mathbb{C})$, we have $\|U_k\|_F^2 = \sum_{j=1}^n \nu_k^{(j)}$ and $\prod_{j=1}^n \nu_k^{(j)} =\det(U_k U_k^\dagger) = \det(U_k) \det(U_k^\dagger) =1$. By the inequality of arithmetic and geometric means, we know 
\begin{equation*}
\frac{1}{n} \sum_{j=1}^n \nu_k^{(j)} \geqslant \left(\prod_{j=1}^n \nu_k^{(j)}\right)^{1/n},
\end{equation*}
and the equality holds if and only if $\nu_k^{(1)}=\nu_k^{(2)}\cdots=\nu_k^{(n)}$. Thus, $\|U_k\|_F^2 \geqslant n$ and the equality holds if and only if all the eigenvalues of $U_k U_k^\dagger$ are $1$. Since any Hermitian matrix whose eigenvalues are all equal to $1$ must be the identity matrix, we conlude that $\|U_k\|_F^2 = n$ if and only if $U_k \in SU(n)$. The conclusion of the theorem is then naturally obtained by summing up the square Frobenius norms of $U_k$.
\end{proof}

By this theorem, we can use $\sqrt{\|\{U\}\|^2 - N n}$ to characterize the distance between $\{U\}$ and $[SU(n)]^N$. Thus finding $\{V\}$ requires to solve the following optimization problem:
\begin{equation} \label{eq:opt}
\argmin_{\{V\} \in [SL(n,\mathbb{C})]^N} \|\{\widetilde{U}\}\|^2.
\end{equation}
Here $\{\widetilde{U}\}$ is defined by \eqref{eq:gc}. This optimization problem is to be solved right after the evolution of the field. More precisely, in the algorithm described in the previous subsection, the following gauge cooling step is to be inserted after steps 2 and 4:
\begin{enumerate}[(GC)]
\item Solve the optimization problem \eqref{eq:opt} and set $U_k \leftarrow V_k^{-1} U_k V_{k+1}$ for all $k = 1,\cdots,N$.
\end{enumerate}
Such a gauge cooling process has been formally justified in \cite{Nagata2016J}. To solve the \eqref{eq:opt} numerically, the gradient descent method is proposed in \cite{Seiler2013}. By assuming
\begin{displaymath}
V_k = \exp \left(\mathrm{i} \sum_{a=1}^{n^2-1} V_{ak} \lambda_a \right), 
\end{displaymath}
where $V_{ak} = X_{ak} + \mathrm{i} Y_{ak}$ and $X_{ak},Y_{ak} \in \mathbb{R}$ for any $a$ and $k$, straightforward calculation yields 
\begin{equation} \label{eq:gradient}
\left\{
\begin{aligned}
&\frac{\partial}{\partial X_{ak}} \|\{\widetilde{U}\}\|^2 \bigg\vert_{\{\widetilde{U}\} = \{U\}}
= 0,  \\
&\frac{\partial}{\partial Y_{ak}} \|\{\widetilde{U}\}\|^2 \bigg\vert_{\{\widetilde{U}\} = \{U\}}
  = 2 \tr[\lambda_a(U_k U_k^{\dagger} - U_{k-1}^{\dagger} U_{k-1})],
\end{aligned}
\right.   \qquad \forall a=1,\cdots,n^2-1, \quad \forall k=1,\cdots,N,
\end{equation}
using which the gradient descent method can be developed. In our work, we are going to show that in the one-dimensional periodic case, such an optimization problem can be solved exactly.

%% file: GaugeCooling.tex
\section{Gauge cooling in the one-dimensional case} \label{sec:GC}
This section is devoted to the exact solution of the optimization problem \eqref{eq:opt}. We are going to start from the one-link case with $N=1$, and then generalize the result to the general one-dimensional case.

\subsection{Introductory study: one-link case}
If $N=1$, there is only one matrix in $\{U\}$. Hence the index $k$ for the matrix, which is always $1$, will be omitted. In this case, the gauge transform \eqref{eq:gc} turns out to be a similarity transform, and the gradient equation \eqref{eq:gradient} can be simplified as
\begin{equation*}
\left\{
\begin{aligned}
&\frac{\partial}{\partial X_a} \|\widetilde{U}\|^2 \bigg\vert_{\widetilde{U} = U}
= 0,\\
&\frac{\partial}{\partial Y_a} \|\widetilde{U}\|^2 \bigg\vert_{\widetilde{U} = U}
= 2 \tr[\lambda_a(U U^{\dagger} - U^{\dagger} U)],
\end{aligned}
\right.  \qquad \forall a=1,\cdots,n^2-1,
\end{equation*}
When the optimal gauge transform is chosen, the above gradient must equal zero for every $a = 1,\cdots,n^2-1$:
\begin{equation} \label{eq:opt_gauge}
\tr[\lambda_a(U U^\dagger - U^\dagger U)] = 0, \qquad \forall a=1,\cdots,n^2-1.
\end{equation}
Since $U U^\dagger-U^\dagger U$ is a Hermitian matrix, there exist real numbers $\alpha_1, \cdots, \alpha_{n^2-1}, \beta$ such that
\begin{equation*}
U U^\dagger-U^\dagger U = \sum_{a=1}^{n^2-1} \alpha_a \lambda_a +\beta I.
\end{equation*}
To determine the coefficients $\alpha_a$, we apply the property \eqref{eq:orth} to get
\begin{displaymath}
\tr[\lambda_b(U U^\dagger - U^\dagger U)] = \sum_{a=1}^{n^2-1} \alpha_a \tr(\lambda_b \lambda_a) + \beta \tr(\lambda_b) = \alpha_b, \qquad b = 1,\cdots,n^2-1,
\end{displaymath}
from which one can see that $\alpha_a = 0$ for all $a = 1,\cdots,n^2-1$ when \eqref{eq:opt_gauge} is fulfilled. Meanwhile, we also have $\beta = 0$ since if $\beta$ were positive, we would get
\begin{displaymath}
\det(U U^\dagger) = \det(U^\dagger U + \beta I) \geqslant \det(U^\dagger U) + \beta \det(I) = \det(U U^\dagger) + \beta,
\end{displaymath}
indicating that $\beta \leqslant 0$, which is a contradiction; a similar contradiction can be obtained if we assume $\beta < 0$. Conclusively, after the optimal gauge transform, the link variable $U$ satisfies $U U^\dagger = U^\dagger U$, which means that $U$ is a normal matrix.

A matrix is a normal matrix if and only if it can be diagonalized by a unitary matrix. Therefore, for any $U \in SL(n,\mathbb{C})$, when $V^{-1} U V$ gives the optimal gauge transform, there exists a diagonal matrix $\Lambda$ and a unitary matrix $P$ such that 
\begin{displaymath}
V^{-1} U V = P^{-1} \Lambda P.
\end{displaymath}
Note that the diagonal entries of $\Lambda$ are the eigenvalues of $U$. When $U$ is diagonalizable, i.e., there exists $Q\in SL(n,\mathbb{C})$ such that $U = Q \Lambda Q^{-1}$, one can find that $V = QP$. Conversely, for a diagonalizable $U$, we can choose an arbitrary unitary matrix $S$, and the gauge transform given by $V = QP$ always transforms $U$ to a normal matrix. A simple choice is $P = I$, for which $\widetilde{U} = V^{-1} U V = \Lambda$.

The above analysis shows that in the one-link case, for every time step, we can choose the gauge transform such that the link $U$ is always a diagonal matrix. Thereby, it is possible to write down a new stochastic process for this diagonal link, and get a clearer picture how gauge cooling helps pull back the complexified links. Before that, we are going to generalize the above results to the 1D case with periodic boundary conditions, and make the analysis more rigorous. Interestingly, in the general 1D case, the link variables $U_k$ can still be maintained as diagonal by choosing optimal gauge transforms. This will be detailed in the following subsection.

\subsection{General 1D case} \label{sec:gc_1D}
Now we are going to restore the index $k$ and study the optimization problem \eqref{eq:opt} for any $N$. To begin with, we show that equating \eqref{eq:gradient} to zero is equivalent to solving the optimization problem:

\begin{lemma}
Suppose $\{U\}$ satisfies
\begin{equation} \label{eq:zero_gradient}
\tr[\lambda_a(U_k U_k^{\dagger} - U_{k-1}^{\dagger} U_{k-1})] = 0,
  \qquad \forall a=1,\cdots,n^2-1, \quad \forall k=1,\cdots,N.
\end{equation}
Then for any gauge transform defined by \eqref{eq:gc}, we have $\|\{\widetilde{U}\}\| \geqslant \|\{U\}\|$.
\end{lemma}

\begin{proof}
It suffices to show that the Hessian matrix is positive semidefinite everywhere, so that the first-order optimality condition \eqref{eq:zero_gradient} gives the global minimum. Again we let $V_{ak} = X_{ak} + \mathrm{i} Y_{ak}$. Because $\partial_{X_{ak}} \|\widetilde{U}\|^2$ equal zero at $\widetilde{U} = U$ for any $U$, we only need to calculate the second-order deriatives of $\|\{\widetilde{U}\}\|^2$ with respect to the imaginary parts $Y_{ak}$. For any $a,b=1,\cdots,n^2-1$ and $ k,l=1,\cdots,N$, by straightforward calcuation, we obtain
\begin{equation*}
H_{akbl}(\{U\}) = \frac{\partial^2}{\partial Y_{ak} \partial Y_{bl}} \|\{\widetilde{U}\}\|^2 \bigg\vert_{\{\widetilde{U}\} = \{U\}}
  = \left\{
\begin{aligned}
  & -4 \tr(\lambda_a U_k \lambda_b U_k^\dagger), \quad &&k=l-1;\\
  & 2 \tr[(\lambda_a \lambda_b + \lambda_b \lambda_a) (U_k U_k^{\dagger} + U_{k-1}^{\dagger} U_{k-1})],
  \quad &&k=l;\\
  & -4 \tr(\lambda_a U_{k-1}^\dagger \lambda_b U_{k-1}), \quad &&k=l+1;\\
  & 0, \quad &&\mathrm{otherwise}.
\end{aligned}
\right.
\end{equation*}
To show the convexity of $\|\{\widetilde{U}\}\|^2$ as a function of $Y_{ak}$, we just need to show that
\begin{displaymath}
R := \sum_{a,b=1}^{n^2-1} \sum_{k,l=1}^N v_{ak} H_{akbl}(\{U\}) v_{bl} \geqslant 0
\end{displaymath}
for any $v_{ak} \in \mathbb{R}$ (note that $H_{akbl}$ is real). To prove this, we calculate $R$ as follows:
\begin{equation*}
\begin{aligned}
R & = \sum_{a,b=1}^{n^2-1} \sum_{k=1}^N v_{ak} H_{akb,k-1}(\{U\}) v_{b,k-1}
  + \sum_{a,b=1}^{n^2-1} \sum_{k=1}^N v_{ak} H_{akb,k+1}(\{U\}) v_{b,k+1}
  + \sum_{a,b=1}^{n^2-1} \sum_{k=1}^N v_{ak} H_{akbk}(\{U\}) v_{bk} \\
& = \sum_{a,b=1}^{n^2-1} \sum_{k,l=1}^N v_{ak} [-4 \tr(\lambda_a U_{k-1}^\dagger \lambda_b U_{k-1})] v_{b,k-1} 
  + \sum_{a,b=1}^{n^2-1} \sum_{k=1}^N v_{ak} [-4 \tr(\lambda_a U_k \lambda_b U_k^\dagger)] v_{b,k+1} \\
& \quad + \sum_{a,b=1}^{n^2-1} \sum_{k=1}^N v_{ak} [2 \tr((\lambda_a \lambda_b + \lambda_b \lambda_a) (U_k U_k^{\dagger} + U_{k-1}^{\dagger} U_{k-1}))] v_{bk} \\
& =-4 \sum_{k=1}^N \tr \left[\big{(}\sum_{a=1}^{n^2-1} \lambda_a v_{ak} \big{)}  U_{k-1}^\dagger \big{(} \sum_{b=1}^{n^2-1} \lambda_b v_{b,k-1} \big{)} U_{k-1} \right] -4 \sum_{k=1}^N \tr \left[\big{(}\sum_{a=1}^{n^2-1} \lambda_a v_{ak} \big{)}  U_k \big{(}\sum_{b=1}^{n^2-1} \lambda_b v_{b,k+1} \big{)}  U_k^\dagger \right] \\
& \quad +2 \sum_{k=1}^N  \tr\left[\left(\big{(}\sum_{a=1}^{n^2-1} \lambda_a v_{ak} \big{)}  \big{(}\sum_{b=1}^{n^2-1} \lambda_b v_{bk} \big{)}  + \big{(}\sum_{b=1}^{n^2-1} \lambda_b v_{bk} \big{)}  \big{(}\sum_{a=1}^{n^2-1} \lambda_a v_{ak} \big{)} \right) (U_k U_k^{\dagger} + U_{k-1}^{\dagger} U_{k-1})\right] .
\end{aligned}
\end{equation*}
For simplicity, we define $M_k = \sum_{a=1}^{n^2-1} v_{ak}\lambda_a$, which is Hermitian, and thus the above equation becomes
\begin{equation*}
\begin{aligned}
R & =-4 \sum_{k=1}^N \tr (M_k  U_{k-1}^\dagger  M_{k-1} U_{k-1})
  -4 \sum_{k=1}^N \tr (M_k U_k M_{k+1}  U_k^\dagger)
  +2 \sum_{k=1}^N  \tr[(M_k  M_k + M_k M_k) (U_k U_k^{\dagger} + U_{k-1}^{\dagger} U_{k-1})] \\
& = 4 \left( \sum_{k=1}^N \tr(M_k  M_k U_{k-1}^{\dagger} U_{k-1})
  -  \sum_{k=1}^N \tr (M_k^\dagger  U_{k-1}^\dagger  M_{k-1} U_{k-1})
  +  \sum_{k=1}^N  \tr (M_k M_k U_k U_k^{\dagger})
  -\sum_{k=1}^N \tr (M_k U_k M_{k+1}  U_k^\dagger) \right).
\end{aligned}
\end{equation*}
Because of the periodic boundary conditions for $\{U\}$, it is obvious that
\begin{equation*}
\begin{aligned}
&\sum_{k=1}^N \tr(M_k  M_k U_{k-1}^{\dagger} U_{k-1})
 = \sum_{k=1}^N \tr(M_{k+1}  M_{k+1} U_k^{\dagger} U_k), \\
& \sum_{k=1}^N \tr (M_k U_{k-1}^\dagger  M_{k-1} U_{k-1})
= \sum_{k=1}^N \tr (M_{k+1}  U_k^\dagger  M_k U_k). 
\end{aligned}
\end{equation*}
Therefore, we have
\begin{equation*}
\begin{aligned}
R & = 4 \left( \sum_{k=1}^N \tr(M_{k+1}  M_{k+1} U_k^{\dagger} U_k)
   - \sum_{k=1}^N \tr (M_{k+1}^\dagger  U_k^\dagger  M_k U_k) 
+  \sum_{k=1}^N  \tr (M_k M_k U_k U_k^{\dagger})
   - \sum_{k=1}^N \tr (M_k^\dagger U_k M_{k+1}  U_k^\dagger] \right) \\
& = 4 \sum_{k=1}^N \tr[(U_k M_{k+1} - M_k U_k)(M_{k+1} U_k^\dagger - U_k^\dagger M_k)]
= 4 \sum_{k=1}^N \tr[(U_k M_{k+1} - M_k U_k)(U_k M_{k+1} - M_k U_k)^\dagger]\\
 & \geqslant 0,
\end{aligned}
\end{equation*}
which indicates that $H_{akbl}(\{U\})$ is positive semidefinite.
\end{proof}
 
By the above lemma, we only need to focus on the algebraic equations \eqref{eq:zero_gradient} with $\{U\}$ replaced by $\{\widetilde{U}\}$. In order to apply the idea used in the one-link case, we also write the gauge transform as a simiarity transform by introducing the block matrix $\mathcal{U}$, defined by
\begin{equation} \label{matrix:U}
   \mathcal{U} =
\begin{pmatrix}
   &U_1 \\
   &&U_2 \\
   &&&\ddots \\
   &&&&U_{N-1} \\
   U_N
\end{pmatrix},
\end{equation}
which is an $nN \times nN$ matrix with $N\times N$ blocks. Thereby, the gauge transform
\begin{equation*}
\widetilde{U}_k = V_k^{-1} U_k V_{k+1}, \quad \forall k=1,\cdots,N
\end{equation*}
can be written as
\begin{equation*}
\widetilde{\mathcal{U}} = \mathcal{V}^{-1} \mathcal{U} \mathcal{V},
\end{equation*}
where $\mathcal{V} = \diag(V_1, V_2, \cdots, V_N)$. Below we claim that similar to the one-link case, the optimal gauge is achieved when $\mathcal{U}$ is a normal matrix.

\begin{theorem} \label{thm:optimality}
The block matrix $\mathcal{U}$ defined by \eqref{matrix:U} is normal if and only if $\{U\}$ satisfies \eqref{eq:zero_gradient}.
\end{theorem}

\begin{proof}
By the same argument as in the one-link case, we know that \eqref{eq:zero_gradient} is equivalent to
\begin{displaymath}
U_k U_k^{\dagger} - U_{k-1}^{\dagger} U_{k-1} = 0, \qquad \forall k = 1,\cdots,N.
\end{displaymath}
Since
\begin{equation*}
\mathcal{U} \mathcal{U}^\dagger - \mathcal{U}^\dagger \mathcal{U} =
\diag (U_1 U_1^\dagger-U_N^\dagger U_N, U_2 U_2^\dagger-U_1^\dagger U_1,\cdots,U_N U_N^\dagger-U_{N-1}^\dagger U_{N-1}^\dagger),
\end{equation*}
the equivalence stated in the theorem is manifest.
\end{proof}

To proceed, we again follow the one-link case and try to find the optimal gauge transform by diagonalizing the matrix $\mathcal{U}$. The result is given in the following theorem:
\begin{theorem}
Suppose $U_N U_1 \cdots U_{N-1}$ can be diagonalized by
\begin{equation} \label{eq:eigUprod}
U_N U_1 \cdots U_{N-1} = Q \Lambda Q^{-1},
\end{equation}
where $\Lambda = \diag(\mu_1, \cdots, \mu_n)$ and $\det Q = 1$. Then $\mathcal{U}$ can be diagonalized by
\begin{displaymath}
\mathcal{U} = \mathcal{T} \mathcal{S} \diag(\Lambda_1, \cdots, \Lambda_N) (\mathcal{T} \mathcal{S})^{-1},
\end{displaymath}
where
\begin{align} \label{eq:T}
\mathcal{T} &= \diag(U_1 U_2 \cdots U_{N-1} Q, U_2 \cdots U_{N-1} Q, \cdots, U_{N-1} Q, Q), \\
\label{eq:Lambda_k}
\Lambda_k &= \diag(\mu_1^{(k)}, \cdots, \mu_n^{(k)}), \qquad k = 1,\cdots,N, \\
\label{eq:S}
\mathcal{S} &= \frac{1}{\sqrt{N}}
\begin{pmatrix}
   \Lambda_1^{-(N-1)} & \Lambda_2^{-(N-1)} & \cdots & \Lambda_N^{-(N-1)} \\
   \Lambda_1^{-(N-2)} & \Lambda_2^{-(N-2)} & \cdots & \Lambda_N^{-(N-2)} \\
   \vdots & \vdots & \ddots & \vdots \\
   \Lambda_1^{-1} & \Lambda_2^{-1} & \cdots & \Lambda_N^{-1} \\
   I & I & \cdots & I
\end{pmatrix},
\end{align}
and $\mu_j^{(1)}, \cdots, \mu_j^{(N)}$ are all the distinct $N$th roots of $\mu_j$.
\end{theorem}

\begin{proof}
We first prove that the diagonal entries of $\Lambda_k$, $k=1,\cdots,N$ provide all the eigenvalues of $\mathcal{U}$. In other words, we are going to show that $\nu$ is an eigenvalue of $\mathcal{U}$ if and only if $\nu^N = \mu_j$ for some $j \in \{1,\cdots,n\}$. If $\nu$ is an eigenvalue of $\mathcal{U}$, since $\det \mathcal{U} = (-1)^{N+1}\det(U_1 U_2 \cdots U_N) = (-1)^{N+1}$, we have $\nu \neq 0$. Let $x = (x_1^T,x_2^T,\cdots,x_N^T)^T$, where $x_k \in \mathbb{C}^n$ for all $k=1,\cdots,N$, be the associated eigenvector. By $\mathcal{U} x = \nu x$, we get
\begin{equation} \label{eq:eigen}
\frac{1}{\nu} U_1 x_2 = x_1,\quad \cdots, \quad \frac{1}{\nu} U_{N-1} x_N = x_{N-1}, \quad \frac{1}{\nu} U_N x_1 = x_N. 
\end{equation}
Concatenating these equations yields
\begin{equation*}
\frac{1}{\nu^N} U_N U_1 \cdots U_{N-1} x_N = x_N
\end{equation*}
which shows that $\nu^N$ is the eigenvalue of $U_N U_1 \cdots U_{N-1}$. On the contrary, for any $j = 1,\cdots,n$, suppose $x_N \in \mathbb{C}^{n} \setminus \{0\}$ is the eigenvector of $U_N U_1 \cdots U_{N-1}$ associated with the eigenvalue $\mu_j$. When $\nu^N = \mu_j$, we can define $x_1, \cdots, x_{N-1}$ by equations \eqref{eq:eigen}. Let $x = (x_1^T,x_2^T,\cdots,x_N^T)^T$, we get $\mathcal{U} x = \nu x$. Therefore, $\nu$ is an eigenvalue of $\mathcal{U}$.

Because $\Lambda_k^N = \Lambda$ for all $k=1,\cdots,N$ based on the equation \eqref{eq:Lambda_k}, it follows from \eqref{eq:eigUprod} that $U_N U_1 \cdots U_{N-1} Q \Lambda_k^{-(N-1)} = Q \Lambda_k$. Using these equality, one obtains by straightforward calculation that
\begin{equation*}
\mathcal{U} \mathcal{T}
\begin{pmatrix}
   \Lambda_k^{-(N-1)} \\
   \Lambda_k^{-(N-2)}  \\
   \vdots  \\
   \Lambda_k^{-1}  \\
   I 
\end{pmatrix}
=
\mathcal{T} 
\begin{pmatrix}
   \Lambda_k^{-(N-1)} \\
   \Lambda_k^{-(N-2)}  \\
   \vdots  \\
   \Lambda_k^{-1}  \\
   I 
\end{pmatrix}
\Lambda_k, \qquad k = 1,\cdots,N.
\end{equation*}
Laying the above equations for all $k=1,\cdots,N$ in a row and dividing the result by $\sqrt{N}$, we get
\begin{equation*}
\mathcal{U} \mathcal{T} \mathcal{S}
=\frac{1}{\sqrt{N}}
\mathcal{T} 
\begin{pmatrix}
   \Lambda_1^{-(N-1)} & \Lambda_2^{-(N-1)} & \cdots & \Lambda_N^{-(N-1)} \\
   \Lambda_1^{-(N-2)} & \Lambda_2^{-(N-2)} & \cdots & \Lambda_N^{-(N-2)} \\
   \vdots & \vdots & \ddots & \vdots \\
   \Lambda_1^{-1} & \Lambda_2^{-1} & \cdots & \Lambda_N^{-1} \\
   I & I & \cdots & I
\end{pmatrix}
\begin{pmatrix}
   \Lambda_1\\
   &\Lambda_2\\
   &&\ddots\\
   &&&\Lambda_N\\
\end{pmatrix}
=\mathcal{T} \mathcal{S} \diag(\Lambda_1, \cdots, \Lambda_n),
\end{equation*}
which concludes the proof.
\end{proof}

By now, it is clear that
\begin{equation} \label{eq:tildeU}
\widetilde{\mathcal{U}} = \mathcal{V}^{-1} \mathcal{T} \mathcal{S} \diag(\Lambda_1, \cdots, \Lambda_N) (\mathcal{T} \mathcal{S})^{-1} \mathcal{V}.
\end{equation}
By Theorem \ref{thm:optimality}, we need to find $\mathcal{V}$ such that the above matrix is normal. This requires us to find $\mathcal{V}$ such that $\mathcal{V}^{-1} \mathcal{T} \mathcal{S}$ is a unitary matrix. Fortunately, one can easily achieve this by the following property of $\mathcal{S}$:
\begin{lemma} \label{lem:S}
The matrix $\mathcal{S}$ defined by \eqref{eq:S} satisfies
\begin{displaymath}
\mathcal{S} \mathcal{S}^\dagger = \diag\left((\Lambda \Lambda^\dagger)^{-\frac{N-1}{N}}, (\Lambda \Lambda^\dagger)^{-\frac{N-2}{N}}, \cdots, (\Lambda \Lambda^\dagger)^{-\frac{1}{N}}, I \right).
\end{displaymath}
\end{lemma}

\begin{proof}
It is easy to get 
\begin{displaymath}
\begin{aligned}
\mathcal{S} \mathcal{S}^\dagger = & \frac{1}{N}
\begin{pmatrix}
   \Lambda_1^{-(N-1)} & \Lambda_2^{-(N-1)} & \cdots & \Lambda_N^{-(N-1)} \\
   \Lambda_1^{-(N-2)} & \Lambda_2^{-(N-2)} & \cdots & \Lambda_N^{-(N-2)} \\
   \vdots & \vdots & \ddots & \vdots \\
   \Lambda_1^{-1} & \Lambda_2^{-1} & \cdots & \Lambda_N^{-1} \\
   I & I & \cdots & I
\end{pmatrix}
\begin{pmatrix}
   (\Lambda_1^{-(N-1)})^\dagger & (\Lambda_1^{-(N-2)})^\dagger & \cdots & (\Lambda_1^{-1})^\dagger & I \\
   (\Lambda_2^{-(N-1)})^\dagger & (\Lambda_2^{-(N-2)})^\dagger & \cdots & (\Lambda_2^{-1})^\dagger & I \\
   \vdots & \vdots & \ddots & \vdots \\
   (\Lambda_N^{-(N-1)})^\dagger & (\Lambda_N^{-(N-2)})^\dagger & \cdots & (\Lambda_N^{-1})^\dagger & I 
\end{pmatrix} \\
= & \frac{1}{N}
\begin{pmatrix}
   \sum_{k=1}^N \Lambda_k^{-(N-1)} (\Lambda_k^{-(N-1)})^\dagger & 
   \sum_{k=1}^N \Lambda_k^{-(N-1)} (\Lambda_k^{-(N-2)})^\dagger & 
   \cdots &
   \sum_{k=1}^N \Lambda_k^{-(N-1)} (\Lambda_k^{-1})^\dagger &
   \sum_{k=1}^N \Lambda_k^{-(N-1)} \\ 
      \sum_{k=1}^N \Lambda_k^{-(N-2)} (\Lambda_k^{-(N-1)})^\dagger & 
   \sum_{k=1}^N \Lambda_k^{-(N-2)} (\Lambda_k^{-(N-2)})^\dagger & 
   \cdots &
   \sum_{k=1}^N \Lambda_k^{-(N-2)} (\Lambda_k^{-1})^\dagger &
   \sum_{k=1}^N \Lambda_k^{-(N-2)} \\ 
      \vdots & \vdots & \ddots & \vdots & \vdots \\
      \sum_{k=1}^N \Lambda_k^{-1} (\Lambda_k^{-(N-1)})^\dagger & 
   \sum_{k=1}^N \Lambda_k^{-1} (\Lambda_k^{-(N-2)})^\dagger & 
   \cdots &
   \sum_{k=1}^N \Lambda_k^{-1} (\Lambda_k^{-1})^\dagger &
   \sum_{k=1}^N \Lambda_k^{-1} \\ 
      \sum_{k=1}^N (\Lambda_k^{-(N-1)})^\dagger & 
   \sum_{k=1}^N (\Lambda_k^{-(N-2)})^\dagger & 
   \cdots &
   \sum_{k=1}^N (\Lambda_k^{-1})^\dagger &
   \sum_{k=1}^N I
\end{pmatrix}.
\end{aligned}
\end{displaymath}
Now we calculate the matrix blocks. By $\Lambda_k^N = \Lambda$, we know that $(\Lambda_k \Lambda_k^\dagger) = (\Lambda \Lambda^\dagger)^\frac{1}{N}$. For any $p,q=0,-1,\cdots,-(N-1)$ and $p \neq q$,
\begin{displaymath}
 \sum_{k=1}^N \Lambda_k^p  (\Lambda_k^p)^\dagger 
 = \sum_{k=1}^N (\Lambda_k \Lambda_k^\dagger)^p 
 =\sum_{k=1}^N  (\Lambda \Lambda^\dagger)^{\frac{p}{N}}
 = N (\Lambda \Lambda^\dagger)^{\frac{p}{N}},
\end{displaymath}
and
\begin{equation} \label{eq:pq}
  \sum_{k=1}^N \Lambda_k^p  (\Lambda_k^q)^\dagger 
  =  \sum_{k=1}^N \Lambda_k^{p-q}  (\Lambda_k  \Lambda_k^\dagger)^q 
  = \left( \sum_{k=1}^N \Lambda_k^{p-q} \right)(\Lambda  \Lambda^\dagger)^{\frac{q}{N}}.
\end{equation}
On the right-hand side of \eqref{eq:pq}, the sum in the parentheses is
\begin{displaymath}
\sum_{k=1}^N \Lambda_k^{p-q} = \diag \left(\sum_{k=1}^N (\mu_1^{(k)})^{p-q},\cdots, \sum_{k=1}^N (\mu_N^{(k)})^{p-q} \right).
\end{displaymath}
To calculate it, for all $j = 1,\cdots,n$, we write $\mu_j$ as $\mu_j = |\mu_j| \exp (\mathrm{i}\theta_j)$, and assume
\begin{displaymath}
\mu_j^{(k)} = |\mu_j|^{\frac{1}{N}} \exp \left(\mathrm{i}\frac{\theta_j+2k\pi}{N}\right), \qquad k = 1,\cdots,N.
\end{displaymath}
Thus
\begin{displaymath}
\begin{aligned}
\sum_{k=1}^N (\mu_j^{(k)})^{p-q}
& = \sum_{k=1}^N |\mu_j|^{\frac{p-q}{N}} \exp \left(\mathrm{i}\frac{(p-q)(\theta_j+2k\pi)}{N}\right)
= |\mu_j|^{\frac{p-q}{N}} \exp \left(\mathrm{i}\frac{p-q}{N} \theta_j\right) 
\sum_{k=1}^N \exp \left[ \left( \mathrm{i}\frac{2(p-q)\pi}{N}\right) k\right] \\
& = |\mu_j|^{\frac{p-q}{N}} \exp \left(\mathrm{i}\frac{p-q}{N} \theta_j\right) 
\frac{\exp \left( \mathrm{i}\frac{2(p-q)\pi}{N}\right) 
 \bigg(1-\exp \left[ \left( \mathrm{i}\frac{2(p-q)\pi}{N}\right) N\right]\bigg)} 
 {1-\exp \left( \mathrm{i}\frac{2(p-q)\pi}{N}\right) }
 = 0.
\end{aligned}
\end{displaymath}
The above equation indicates that $\sum_{k=1}^N \Lambda_k^p  (\Lambda_k^q)^\dagger = 0$. Therefore
\begin{align*}
\mathcal{S} \mathcal{S}^\dagger 
& =\frac{1}{N} \diag \left( N(\Lambda \Lambda^\dagger)^{\frac{-(N-1)}{N}}, N(\Lambda \Lambda^\dagger)^{-\frac{N-2}{N}}, \cdots, N(\Lambda \Lambda^\dagger)^{-\frac{1}{N}}, 
NI \right) \\
& = \diag\left((\Lambda \Lambda^\dagger)^{-\frac{N-1}{N}}, (\Lambda \Lambda^\dagger)^{-\frac{N-2}{N}}, \cdots, (\Lambda \Lambda^\dagger)^{-\frac{1}{N}}, I \right). \qedhere
\end{align*}
\end{proof}

Based on Lemma \ref{lem:S}, we provide in the following theorem a series of convenient choices of optimal gauge transforms, which are obviously a generalization of the one-link case:

\begin{theorem} \label{thm:opt_cooling}
Suppose $U_N U_1 \cdots U_{N-1}$ is diagonalizable. For any $P_k \in SL(n,\mathbb{C})$, $k=1,\cdots,N$, define
\begin{equation} \label{eq:V_k}
V_k=\left\{
\begin{aligned}
&(U_k \cdots U_{N-1})Q(\Lambda \Lambda^\dagger)^{-\frac{N-k}{2N}} P_k, &k&=1,\cdots,N-1,\\
&Q P_k, &k&=N.
\end{aligned}
\right.
\end{equation}
We have that $\mathcal{V}^{-1} \mathcal{T} \mathcal{S}$ is unitary. In this case,
\begin{displaymath}
\|\{\widetilde{U}\}\| = \sqrt{N \left(\sum_{j=1}^n |\mu_j|^\frac{2}{N} \right)}.
\end{displaymath}
\end{theorem}

\begin{proof}
With $\mathcal{V} = \diag(V_1,\cdots,V_N)$, it is easy to calculate
\begin{displaymath}
\mathcal{V}^{-1} \mathcal{T} \mathcal{S} (\mathcal{V}^{-1} \mathcal{T} \mathcal{S})^\dagger = \mathcal{V}^{-1} \mathcal{T} (\mathcal{S} \mathcal{S}^\dagger) \mathcal{T}^\dagger \mathcal{V}^{-\dagger} = I  
\end{displaymath}
by Lemma \ref{lem:S} and the equation \eqref{eq:T}.
 
Due to $\mathcal{U} \mathcal{U}^\dagger = \diag(U_1 U_1^\dagger, \cdots, U_N U_N^\dagger)$, we have
\begin{displaymath}
\|\{\widetilde{U}\}\|^2 = \sum_{k=1}^N \tr(\widetilde{U}_k \widetilde{U}_k^\dagger) = \tr(\widetilde{\mathcal{U}} \widetilde{\mathcal{U}}^\dagger).
\end{displaymath}
For $\widetilde{\mathcal{U}} = \mathcal{V}^{-1} \mathcal{T} \mathcal{S} \diag(\Lambda_1, \cdots, \Lambda_N) (\mathcal{V}^{-1} \mathcal{T} \mathcal{S})^\dagger$ with the equation \eqref{eq:tildeU}, there is 
\begin{equation*}
\tr(\widetilde{\mathcal{U}} \widetilde{\mathcal{U}}^\dagger)
= \tr(\Lambda_1 \Lambda_1^\dagger, \cdots, \Lambda_N \Lambda_N^\dagger)
= \tr \left(\sum_{k=1}^N \Lambda_k \Lambda_k^\dagger\right)
= \tr(N(\Lambda \Lambda^\dagger)^\frac{1}{N})
= N \sum_{j=1}^n |\mu_j|^\frac{2}{N}.
\end{equation*}
Yet,
\begin{displaymath}
\|\{\widetilde{U}\}\| = \sqrt{N \left(\sum_{j=1}^n |\mu_j|^\frac{2}{N} \right)}. \qedhere
\end{displaymath}
\end{proof}

In practice, we can simply choose $P_k = I$, and thus
\begin{equation} \label{eq:diagonal}
\widetilde{U}_k = V_k^{-1} U_k V_{k+1}=\left\{
\begin{aligned}
&(\Lambda \Lambda^\dagger)^\frac{1}{2N}, ~ &k&=1,\cdots,N-1,\\
&\Lambda(\Lambda \Lambda^\dagger)^{-\frac{N-1}{2N}}, ~ &k&=N,
\end{aligned}
\right.
\end{equation}
from which we can see that by choosing appropriate gauge transform, all the links are always diagonal. This shows us how the gauge cooling technique helps prevent the field links from excursing too far away from $SU(n)$. Gauge cooling removes redundant degrees of freedom, so that the probability of imaginary excursion is greatly reduced. Therefore gauge cooling can be considered as a method of dimension reduction to enhance the numerical stability. Here we would like to comment that in the multidimensional case, the exact solution no longer has the simple diagonal form as \eqref{eq:diagonal}. But one can still expect the similar effect. In the next section, we are going to present the effect of gauge cooling on the drift term. Before that, we will first pick up a special case left out in the above discussion.

By now, we have only considered the case when the matrix $U_N U_1 \cdots U_{N-1}$ is diagonalizable. When it is not diagonalizable, we consider its Jordan normal form given by
\begin{equation*}
   U_N U_1 \cdots U_{N-1} = Q'
\begin{pmatrix}
   \mu_1&\delta_1 \\
   &\mu_2&\delta_2 \\
   &&\ddots&\ddots \\
   &&&\mu_{n-1}&\delta_{n-1} \\
   &&&&\mu_n
\end{pmatrix}Q'^{-1}
\end{equation*}
where $\mu_m, m=1,2,\cdots,n$ are eigenvalues of $U_N U_1 \cdots U_{N-1}$, and $\delta_m, m=1,2,\cdots,n-1$ are $0$ or $1$. At least one of $\delta_m$, $m = 1,2,\cdots,n-1$ is nonzero. In this case, we choose a positive number $\epsilon \ll 1$ and define $Q = \diag \{1,\epsilon,\cdots,\epsilon^{n-1}\}Q'$. Thus
\begin{equation*}
   U_N U_1 \cdots U_{N-1} = Q
\begin{pmatrix}
   \mu_1&\epsilon \delta_1 \\
   &\mu_2&\epsilon \delta_2 \\
   &&\ddots&\ddots \\
   &&&\mu_{n-1}&\epsilon \delta_{n-1} \\
   &&&&\mu_n
\end{pmatrix}Q^{-1},
\end{equation*}
where the middle matrix on the right-hand side is very close to a diagonal matrix. By scaling $Q'$ properly such that $\det Q = 1$ and choosing $V_k$ according to \eqref{eq:V_k}, the matrix $\widetilde{\mathcal{U}}$ can be very close to a normal matrix. Since $\epsilon$ can be arbitrarily small, the norm $\|\{\widetilde{U}\}\|$ can be arbitrarily close to $\sqrt{N \left(\sum_{j=1}^n |\mu_j|^\frac{2}{N} \right)}$, and the new links $\{\widetilde{U}\}$ can be arbitrarily close to \eqref{eq:diagonal}. In fact, this is the situation when the optimal gauge transform does not exist, and $\sqrt{N \left(\sum_{j=1}^n |\mu_j|^\frac{2}{N} \right)}$ is the infimum of all possible values of $\|\{\widetilde{U}\}\|$. However, in $SL(n,\mathbb{C})$, nondiagonalizable matrices only form a submanifold with zero measure. Therefore, in our implementation, we just set the new matrices to be \eqref{eq:diagonal} without checking the diagonalizablity of $U_N U_1 \cdots U_{N-1}$.

%% file: Eigenvalue.tex
\section{Study of gauge cooling by stochastic differential equations} \label{sec:eig}
To further explore the effect of gauge cooling, we will fix the choice of gauge transforms as \eqref{eq:diagonal} for all time steps. Thereby, all the links are always diagonal matrices, and in fact, these links have only $n-1$ degrees of freedom, which are $\mu_1, \cdots, \mu_{n-1}$ (note that $\mu_1 \cdots \mu_n = 1$). Thus, it is helpful to study the dynamics of these quantities directly, by taking an infinitesimal time step. Below we will first give the general formulation for the $SU(n)$ case, and then focus on the special action \eqref{eq:action} for the $SU(2)$ case.

\subsection{General case in $SU(n)$} \label{sec:SU(n)}
Our analysis in Section \ref{sec:gc_1D} shows that optimal gauge cooling can maintain the following structure of $\{U\}$:
\begin{equation*}
U_k = \left\{
\begin{aligned}
&(\Lambda \Lambda^\dagger)^\frac{1}{2N}, ~ &k&=1,\cdots,N-1,\\
&\Lambda(\Lambda \Lambda^\dagger)^{-\frac{N-1}{2N}}, ~ &k&=N,
\end{aligned}
\right.
\end{equation*}
where $\Lambda = \diag(\mu_1,\cdots,\mu_N)$. Before deriving the dynamics of $\mu_j$, we first review the steps to evolve the links by one time step $\Delta t$:
\begin{enumerate}[1.]
\item Let $U_k^* = \exp \left(-\sum_{a=1}^{n^2-1} \mathrm{i} \lambda_a (K_{ak} \Delta t + \eta_{ak} \sqrt{2\Delta t}) \right) U_k$, where $K_{ak}= D_{ak}S(\{U\})$ and $\eta_{ak}$ obeys the standard normal distribution.
\item Compute the eigenvalues of $U_N^* U_1^* \cdots U_{N-1}^*$, and use $\mu_1^*, \cdots, \mu_N^*$ to denote the results.
\item Update the matrices by setting $U_1, \cdots, U_{N-1}$ to be $\diag(|\mu_1^*|^{\frac{1}{N}}, \cdots, |\mu_n^*|^{\frac{1}{N}})$, and setting $U_N$ to be $\diag(|\mu_1^*|^{-\frac{N-1}{N}} \mu_1^*, \cdots, |\mu_n^*|^{-\frac{N-1}{N}} \mu_n^*)$.
\end{enumerate}
Here the first step is identical to the evolution without gauge cooling, and the second and third steps are an application of equation \eqref{eq:diagonal}, which performs the optimal gauge cooling. Next, we are going to replace $\Delta t$ in the above equation by an infinitesimal time step $\mathrm{d}t$, and $\eta_{ak} \sqrt{2\Delta t}$ will be replaced by $\mathrm{d}w_{ak}$ correspondingly. Such a limit can be more easily taken if we consider the above process as the Euler-Maruyama method for It\^o SDE (see e.g. \cite[Section 3.6]{Grinfeld2014} for the description of numerical methods for It\^o and Stratonovich SDEs). Therefore in this section, we will switch to It\^o's representation, where the circle sign ``$\circ$'' in front of $\mathrm{d}w_{ak}$ is removed. For easier understanding, below we are going to perform a formal derivation of the SDEs for $\mu_j$, $j = 1,\cdots,n$, where notations such as $\mathrm{d}t$ and $\mathrm{d}w_{ak}$ will be used directly to replace $\Delta t$ and $\eta_{ak} \sqrt{2\Delta t}$, and all the terms with magnitude smaller than $\mathrm{d}t$ will be discarded without explicit approximation. We claim that such a process can be made rigorous following the classical error analysis for the Euler-Maruyama method without much difficulty \cite{Kloeden1992}.

We first write down the infinitesimal version of Step 1 using It\^o calculus:
\begin{equation*}
\begin{split}
  U_k^* 
  & = \exp \left(-\sum_{a=1}^{n^2-1} \mathrm{i} \lambda_a (K_{ak}\,\mathrm{d}t + \mathrm{d}w_{ak}) \right) U_k
    = \left( I - \sum_{a=1}^{n^2-1} \mathrm{i} \lambda_a (K_{ak}\,\mathrm{d}t + \mathrm{d}w_{ak}) - \sum_{a=1}^{n^2-1} \lambda_a^2 \,\mathrm{d}t \right) U_k \\
  & = \left( I - \sum_{a=1}^{n^2-1} \mathrm{i} \lambda_a (K_{ak}\,\mathrm{d}t + \mathrm{d}w_{ak}) - \frac{2(n^2-1)}{n} I \,\mathrm{d}t \right) U_k, \qquad k = 1,\cdots,N,
\end{split}
\end{equation*}
where we have used $(\mathrm{d} w_{ak})^2 = 2\,\mathrm{d}t$ and
\begin{displaymath}
\sum_{a=1}^{n^2-1} \lambda_a^2 = \frac{2(n^2-1)}{n} I,
\end{displaymath}
which can be found in \cite{Greiner2007}. Step 2 requires us to compute the product of these matrices:
\begin{equation} \label{eq:prod_U*}
\begin{split}
U_N^* U_1^* \cdots U_{N-1}^* &=
  \left( I - \sum_{a=1}^{n^2-1} \mathrm{i} \lambda_a (K_{aN}\mathrm{d}t + \mathrm{d}w_{aN}) -  \frac{2(n^2-1)}{n} I \mathrm{d}t  \right) \Lambda(\Lambda \Lambda^\dagger)^{-\frac{N-1}{2N}} \\
  & \qquad \times \left( I - \sum_{a=1}^{n^2-1} \mathrm{i} \lambda_a (K_{a1}\mathrm{d}t + \mathrm{d}w_{a1}) -  \frac{2(n^2-1)}{n}  I \mathrm{d}t \right) (\Lambda \Lambda^\dagger)^\frac{1}{2N} \cdots \\
  & \qquad \times \left( I - \sum_{a=1}^{n^2-1} \mathrm{i} \lambda_a (K_{a,N-1}\mathrm{d}t + \mathrm{d}w_{a,N-1} ) -  \frac{2(n^2-1)}{n} I \mathrm{d}t \right) (\Lambda \Lambda^\dagger)^\frac{1}{2N} \\
&= \Lambda - \mathrm{i} \sum_{a=1}^{n^2-1} \sum_{k=1}^N W_{ak} \,\mathrm{d}w_{ak} - \left(
  \mathrm{i} \sum_{a=1}^{n^2-1} \sum_{k=1}^N K_{ak} W_{ak} + \frac{2N(n^2-1)}{n} \Lambda
\right) \mathrm{d}t,
\end{split}
\end{equation}
where
\begin{equation*}
W_{ak} = \left\{
\begin{aligned}
&\Lambda (\Lambda \Lambda^\dagger)^{-\frac{N-k}{2N}} \lambda_a (\Lambda \Lambda^\dagger)^\frac{N-k}{2N} , \quad &k&=1,\cdots,N-1, \\
&\lambda_a \Lambda, \quad &k&=N.
\end{aligned}
\right.
\end{equation*}
Next, we need to find all its eigenvalues. Suppose $r_j^*$ is the unit eigenvector associated with the eigenvalue $\mu_j^*$:
\begin{equation} \label{eq:eigen*}
U_N^* U_1^* \cdots U_{N-1}^* r_j^* = \mu_j^* r_j^*, \qquad j = 1,\cdots,n.
\end{equation}
By \eqref{eq:prod_U*}, as $U_N^* U_1^* \cdots U_{N-1}^*$ is a perturbation of the diagonal matrix $\Lambda$, the eigenvalue $\mu_j^*$ should be the perturbation of the diagonal entry $\mu_j$, and $r_j^*$ should be the perturbation of the canonical basis $e_j$. Assume
\begin{equation} \label{eq:ev}
\mu_j^* =  \mu_j + \sum_{a=1}^{n^2-1} \sum_{k=1}^{N} c_j^{ak} \mathrm{d} w_{ak} + \sigma_j \mathrm{d}t, 
\quad r_j^* = e_j + \sum_{a=1}^{n^2-1} \sum_{k=1}^{N} h_j^{ak} \mathrm{d} w_{ak} + p_j \mathrm{d}t,
\qquad j = 1,\cdots,n.
\end{equation}
Substituting \eqref{eq:ev} and \eqref{eq:prod_U*} into \eqref{eq:eigen*} yields
\begin{equation*} 
\begin{aligned}
& \sum_{a=1}^{n^2-1} \sum_{k=1}^{N} (\Lambda h_j^{ak} -  \mathrm{i} W_{ak} e_j) \,\mathrm{d}w_{ak} + \left( \Lambda p_j - \mathrm{i} \sum_{a=1}^{n^2-1} \sum_{k=1}^{N} K_{ak} W_{ak} e_j -2 \mathrm{i} \sum_{a=1}^{n^2-1} W_{ak}h_j^{ak} -\frac{2N(n^2-1)}{n}  \Lambda e_j \right) \mathrm{d}t \\
=& \sum_{a=1}^{n^2-1} \sum_{k=1}^{N} (\mu_j h_j^{ak} + c_j^{ak} e_j) \,\mathrm{d}w_{ak} + \left( 2 \sum_{a=1}^{n^2-1}\sum_{k=1}^{N} c_j^{ak} h_j^{ak} + \sigma_j e_j + \mu_j p_j \right) \mathrm{d}t.
\end{aligned}
\end{equation*}
Now we equate the coefficients of the same infinitesimal terms to get
\begin{equation} \label{eq:eigdw}
\Lambda h_j^{ak} - \mathrm{i} W_{ak} e_j = \mu_j h_j^{ak} + c_j^{ak} e_j,\quad \forall a,k,
\end{equation}
and
\begin{equation} \label{eq:eigdt}
\Lambda p_j - \mathrm{i} \sum_{a=1}^{n^2-1} \sum_{k=1}^{N} K_{ak} W_{ak} e_j -2 \mathrm{i} \sum_{a=1}^{n^2-1} W_{ak}h_j^{ak} -\frac{2N(n^2-1)}{n} \Lambda e_j
= 2 \sum_{a=1}^{n^2-1}\sum_{k=1}^{N} c_j^{ak} h_j^{ak} + \sigma_j e_j + \mu_j p_j, \quad \forall a,k.
\end{equation}
For the equation \eqref{eq:eigdw}, we multiply both sides by $e_j^T$ on the left, and obtain
\begin{equation*}
e_j^T \Lambda h_j^{ak} - \mathrm{i} \mu_j e_j^T \lambda_a e_j 
= \mu_j e_j^T h_j^{ak} + c_j^{ak}.
\end{equation*}
Note that $e_j^T \Lambda = e_j^T \mu_j$. We get $c_j^{ak}$ by cancelling out the terms $\mu_j e_j^T h_j^{ak}$ on both sides:
\begin{equation} \label{coff:dw}
c_j^{ak} = - \mathrm{i} \mu_j e_j^T \lambda_a e_j.
\end{equation}
Because $c_j^{ak}$ is independent of $k$, we rewrite it as $c_j^a$ hereafter. If we multiply both sides of \eqref{eq:eigdw} by $e_l^T$ on the left for $l \neq j$, the result is
\begin{equation*}
(\mu_j-\mu_l) e_l^T h_j^{ak} = - \mathrm{i} e_l^T W_{ak} e_j.
\end{equation*}
When $\mu_j \neq \mu_l$, we have
\begin{equation*}
(h^{ak}_j)_l = \frac{\mathrm{i} e_l^T W_{ak} e_j}{\mu_l - \mu_j}
=\left\{
\begin{aligned}
&\frac{\mathrm{i} |\mu_j|^\frac{N-k}{N} |\mu_l|^{-\frac{N-k}{N}} \mu_l e_l^T \lambda_a e_j}{\mu_l - \mu_j}, \quad &k&=1,\cdots,N-1;\\
&\frac{\mathrm{i} \mu_j e_l^T \lambda_a e_j}{\mu_l - \mu_j}, \quad &k&=N;
\end{aligned}
\right.   \qquad l \neq j, 
\end{equation*}
where $(h^{ak}_j)_l$ denotes the $l$th component of $h^{ak}_j$. In order to determine the $j$th component of $h^{ak}_j$, we impose the following two conditions:
\begin{enumerate}
\item $(r_j^*)^\dagger r_j^* = 1$;
\item The $j$th component of $r_j^*$ is real.
\end{enumerate}
Since $r_j^*$ is a perturbation of $e_j$, the $j$th component of $r_j^*$ must be nonzero. Thus the above two constraints fix the eigenvector. By the first constraint, we have
\begin{equation*}
\begin{aligned}
1 = (r_j^*)^\dagger r_j^*
& = \left(e_j + \sum_{a=1}^{n^2-1} \sum_{k=1}^{N} h_j^{ak} \mathrm{d} w_{ak} + p_j \mathrm{d}t \right)^\dagger \left(e_j + \sum_{a=1}^{n^2-1} \sum_{k=1}^{N} h_j^{ak} \mathrm{d} w_{ak} + p_j \mathrm{d}t \right) \\
& = 1 + 2 \sum_{a=1}^{n^2-1} \re\left(\sum_{k=1}^{N} e_j^T h_j^{ak}\right) \mathrm{d} w_{ak}
+ \left(2 \re(e_j^T p_j) + 2 \sum_{a=1}^{n^2-1} \sum_{k=1}^{N} (h_j^{ak})^\dagger h_j^{ak} \right) \mathrm{d}t.
\end{aligned}
\end{equation*}
Therefore,
\begin{equation*}
\re(e_j^Th_j^{ak}) = 0, \quad \text{i.e.} \quad \re((h_j^{ak})_j)=0.
\end{equation*}
The second constraint requires that the imaginary part of $(h_j^{ak})_j$ be zero as well. Thus $(h_j^{ak})_j=0$. Now, we can write $h_j^{ak}$ explicitly as
\begin{equation*}
\begin{aligned}
h_j^{ak} & =  \mathrm{i} |\mu_j|^\frac{N-k}{N} \bigg{(} \frac{ |\mu_1|^{-\frac{N-k}{N}} \mu_1 e_1^T \lambda_a e_j}{\mu_1 - \mu_j}, \cdots,  
\frac{|\mu_{j-1}|^{-\frac{N-k}{N}} \mu_{j-1} e_{j-1}^T \lambda_a e_j}{\mu_{j-1} - \mu_j}, 0 , \\
& \qquad \qquad \qquad \qquad \frac{ |\mu_{j+1}|^{-\frac{N-k}{N}} \mu_{j+1} e_{j+1}^T \lambda_a e_j}{\mu_{j+1} - \mu_j}, \cdots, 
\frac{|\mu_n|^{-\frac{N-k}{N}} \mu_n e_n^T \lambda_a e_j}{\mu_n - \mu_j} \bigg{)}^T\\
& = \mathrm{i} |\mu_j|^\frac{N-k}{N} (\Lambda \Lambda^\dagger)^{-\frac{N-k}{2N}} \Lambda \diag \bigg{(}\frac{1}{\mu_1 - \mu_j}, \cdots,  \frac{1}{\mu_{j-1} - \mu_j}, 0 , \frac{1}{\mu_{j+1} - \mu_j}, \cdots, \frac{1}{\mu_n - \mu_j} \bigg{)} \lambda_a e_j,
\end{aligned}
\end{equation*}
where $k=1,\cdots,N-1$, and
\begin{equation*}
\begin{aligned}
h_j^{aN} & = \mathrm{i} \mu_j \bigg{(} \frac{e_1^T \lambda_a e_j}{\mu_1 - \mu_j},\cdots, \frac{ e_{j-1}^T \lambda_a e_j}{\mu_{j-1} - \mu_j}, 0 ,
 \frac{e_{j+1}^T \lambda_a e_j}{\mu_{j+1} - \mu_j}, \cdots, 
 \frac{e_N^T \lambda_a e_j}{\mu_N - \mu_j} \bigg{)}^T \\
& = \mathrm{i} \mu_j \diag \bigg{(} \frac{1}{\mu_1 - \mu_j}, \cdots, \frac{1}{\mu_{j-1} - \mu_j}, 0, \frac{1}{\mu_{j+1} - \mu_j}, \cdots, \frac{1}{\mu_N - \mu_j} \bigg{)} \lambda_a e_j.
\end{aligned}
\end{equation*}
By these results, we can find $\sigma_j$ by multiplying both sides of \eqref{eq:eigdt} by $e_j^T$ on the left:
\begin{equation} \label{eq:sigma_j}
\begin{aligned}
\sigma_j  & = - \mathrm{i} \mu_j \sum_{a=1}^{n^2-1} \sum_{k=1}^{N} K_{ak}  e_j^T \lambda_a e_j - 2 \mathrm{i} \sum_{a=1}^{n^2-1} \sum_{k=1}^{N} e_j^T W_{ak} h^{ak}_j - \frac{2N(n^2-1)}{n} \mu_j \\
&  = - \mathrm{i} \mu_j \sum_{a=1}^{n^2-1} \sum_{k=1}^{N} K_{ak}  e_j^T \lambda_a e_j + 2 N \mu_j \sum_{a=1}^{n^2-1} e_j^T \lambda_a \Omega_j \lambda_a e_j - \frac{2N(n^2-1)}{n} \mu_j,
\end{aligned}
\end{equation}
where 
\begin{displaymath}
\Omega_j = \diag \left( \frac{\mu_1}{\mu_1 - \mu_j}, \cdots, \frac{\mu_{j-1}}{\mu_{j-1} - \mu_j}, 0, \frac{\mu_{j+1}}{\mu_{j+1} - \mu_j}, \cdots, \frac{\mu_N}{\mu_N - \mu_j} \right).
\end{displaymath}
To summarize, we insert \eqref{coff:dw} and \eqref{eq:sigma_j} into \eqref{eq:ev}, which gives the value of $\mu_j^*$:
\begin{equation*}
\mu_j^* =  \mu_j - \sum_{a=1}^{n^2-1} \sum_{k=1}^{N} \mathrm{i} \mu_j e_j^T \lambda_a e_j \,\mathrm{d} w_{ak} 
- \left[  \mathrm{i} \mu_j \sum_{a=1}^{n^2-1} \sum_{k=1}^{N} K_{ak}  e_j^T \lambda_a e_j - 2 N \mu_j \sum_{a=1}^{n^2-1} e_j^T \lambda_a \Omega_j \lambda_a e_j + \frac{2N(n^2-1)}{n} \mu_j \right] \mathrm{d}t.
\end{equation*}
By now, Step 2 is accomplished.

Step 3 simply means that $\mu_j^*$ is just $\mu_j + \mathrm{d}\mu_j$. Hence,
\begin{equation} \label{eq:dmu}
\mathrm{d} \mu_j = - \sum_{a=1}^{n^2-1} \sum_{k=1}^{N} \mathrm{i} \mu_j e_j^T \lambda_a e_j \,\mathrm{d} w_{ak} 
- \left[ \mathrm{i} \mu_j \sum_{a=1}^{n^2-1} \sum_{k=1}^{N} K_{ak}  e_j^T \lambda_a e_j -2 N \mu_j \sum_{a=1}^{n^2-1} e_j^T \lambda_a \Omega_j \lambda_a e_j + \frac{2N(n^2-1)}{n} \mu_j \right] \mathrm{d}t.
\end{equation}
To make further simplification, we define
\begin{displaymath}
\mathrm{d}w_a = \frac{1}{N} \sum_{k=1}^N \mathrm{d}w_{ak},
  \quad K_a = \frac{1}{N} \sum_{k=1}^N K_{ak}, \qquad a = 1,\cdots,n^2-1.
\end{displaymath}
Thus \eqref{eq:dmu} can be rewritten as
\begin{equation} \label{eq:dmu_final}
\frac{1}{N} \mathrm{d} \mu_j = -\sum_{a=1}^{n^2-1} \mathrm{i} \mu_j e_j^T \lambda_a e_j \,\mathrm{d} w_a
- \left[ \mathrm{i} \mu_j \sum_{a=1}^{n^2-1} K_a e_j^T \lambda_a e_j - 2 \mu_j \sum_{a=1}^{n^2-1} e_j^T \lambda_a \Omega_j \lambda_a e_j + \frac{2(n^2-1)}{n} \mu_j \right] \mathrm{d}t.
\end{equation}
This equation shows that by gauge cooling, the one-dimensional $N$-link model is essentially equivalent to the one-link model, as is known for the exact solution \cite{Seiler2013}. The optimal gauge cooling automatically utilizes this property, and therefore greatly reduces the difficulty of the simulation. This analysis further confirms the nature of dimension reduction for the gauge cooling technique from another point of view, which also applies in the multi-dimensional cases.

Some other interesting properties can be observed from \eqref{eq:dmu_final}. For example, we can see that only $n-1$ drift terms are active for the Langevin process. This is clear for $n=2$ since the generators of $SU(2)$ are just Pauli matrices (below we add the superscript $(n)$ to matrices $\lambda_a$ for clarification):
\begin{equation*}
   \lambda_1^{(2)} = 
\begin{pmatrix}
   0&1 \\
   1&0
\end{pmatrix},\quad
   \lambda_2^{(2)} = 
\begin{pmatrix}
   0&\mathrm{i} \\
   -\mathrm{i}&0
\end{pmatrix},\quad
   \lambda_3^{(2)} = 
\begin{pmatrix}
   1&0 \\
   0&-1
\end{pmatrix},
\end{equation*}
and we have $e_j^T \lambda_1^{(2)} e_j \equiv 0$ and $e_j^T \lambda_2^{(2)} e_j \equiv 0$, which means that both $K_1$ and $K_2$ are multiplied by zero in \eqref{eq:dmu_final}. When $n=3$, the generators of $SU(3)$ are Gell-Mann matrices:
\begin{equation} \label{eq:Gell-Mann}
\begin{gathered}
\lambda_a^{(3)} = 
\begin{pmatrix}
   \lambda_a^{(2)} \\
   &0
\end{pmatrix},\quad a=1,2,3, \\
  \lambda_4^{(3)} \!=\!
\begin{pmatrix}
   0&0&1\\
   0&0&0\\
   1&0&0
\end{pmatrix}, ~~
  \lambda_5^{(3)} \!=\!
\begin{pmatrix}
   0&0&\mathrm{i}\\
   0&0&0\\
   -\mathrm{i}&0&0
\end{pmatrix}, ~~
   \lambda_6^{(3)} \!=\!
\begin{pmatrix}
   0&0&0\\
   0&0&1\\
   0&1&0
\end{pmatrix}, ~~
   \lambda_7^{(3)} \!=\!
\begin{pmatrix}
   0&0&0\\
   0&0&\mathrm{i}\\
   0&-\mathrm{i}&0
\end{pmatrix}, ~~
   \lambda_8^{(3)} \!= \frac{1}{\sqrt{3}}
\begin{pmatrix}
   1&0&0\\
   0&1&0\\
   0&0&-2
\end{pmatrix},
\end{gathered}
\end{equation}
from which one sees that $e_j^T \lambda_a^{(3)} e_j \equiv 0$ if $a \neq 3$ and $a \neq 8$. Therefore only $K_3$ and $K_8$ are active. Likewise, for $n\geqslant 3$, the first $n(n-2)$ generators of $SU(n)$ can be chosen as $\lambda_a^{(n)} = \diag(\lambda_a^{(n-1)}, 0)$. For the rest generators, the only one with nonzero diagonal entries is \cite{Georgi1999}
\begin{displaymath}
\lambda_{n^2-1}^{(n)} = \sqrt{\frac{2}{n(n-1)}} \diag(1,\cdots,1,1-n).
\end{displaymath}
Therefore, by mathematical induction, the number of active drift terms in \eqref{eq:dmu_final} equals $n-1$.

In \eqref{eq:dmu_final}, except the term with $K_a$, the other two drift terms are independent of the action $S(\{U\})$. These terms also help us stabilize the complex Langevin process, as will be demonstrated by a specific example in the next subsection.

\subsection{Polyakov loop model for the $SU(2)$ theory} \label{sec:SU(2)}
This subsection will be devoted to a special case study. For easier demonstration, we choose $n=2$ since in this case, there is only one degree of freedom, namely $\mu_1$ (or $\mu_2$), as stated in the begining of Section \ref{sec:eig}. By taking this choice, it is much easier to understand the complex Langevin dynamics, since one can explicitly plot the flow field and the distribution of samples on the complex plane, giving a clear illustration of possible issues. Such a method has been widely used in existing investigations of the complex Langevin method \cite{Aarts2013,Nagata2016,Scherzer2019}.

Based on the discussion of the previous subsection, we just need to consider the one-link model with $N=1$. Since $\mu_1 \mu_2 = 1$, here we only study the dynamics of $\mu_1$. For simplicity, we omit the subscript and write $\mu_1$ as $\mu$. Thus, by equation \eqref{eq:dmu}, we can write down the SDE for $\mu$ as follows:
\begin{equation} \label{eq:dmu2}
\mathrm{d}\mu = - \mathrm{i} \mu \,\mathrm{d}w_3 - \left( \mathrm{i} \mu K_3 + \frac{4 \mu}{\mu^2-1} + 3\mu \right) \mathrm{d}t.
\end{equation}
For such a representation, $U \in SU(2)$ corresponds to $\mu \in U(1)$. Therefore it is natural to make the change of variable $\mu = \exp(-\mathrm{i} s)$ and consider the SDE for $s$, which is
\begin{equation} \label{eq:ds2}
\mathrm{d}s = \left( K_3 + 2\mathrm{i} \, \frac{1+\mu^2}{1-\mu^2} \right) \,\mathrm{d}t + \mathrm{d}w_3
  = \left( K_3 + 2 \cot s \right) \,\mathrm{d}t + \mathrm{d}w_3.
\end{equation}
The above equation can be verified by It\^o's calculus:
\begin{displaymath}
\mathrm{d}\mu = \mathrm{d} \big( \exp(-\mathrm{i}s) \big) = -\mathrm{i} \mu \,\mathrm{d}s - \mu \,\mathrm{d}t
  = -\mathrm{i} \mu \,\mathrm{d}w_3 - \mu \left( \mathrm{i} K_3 - \frac{2(1+\mu^2)}{1-\mu^2} + 1 \right) \,\mathrm{d}t,
\end{displaymath}
which agrees with \eqref{eq:dmu2}. Note that $\mu$ is $2\pi$-periodic with respect to the real part of $s$. Therefore in \eqref{eq:ds2}, we need to consider $s(t)$ as a stochastic process defined on the cylinder $(\re s(t), \im s(t)) \in \mathbb{S}^1 \times \mathbb{R}$. When $K_3 = 0$, the flow field given by the drift term is plotted in Figure \ref{Fig:flow_k0}, from which we can see that when $s$ is not real, the flow pushes $s$ toward the real axis, meaning that the cotangent term in \eqref{eq:ds2} also helps stabilize the complex Langevin method. This also implies that when $K_3$ does not provide a strong repulsion from the real axis, the variable $s$ can stay in a region near the real axis, leading to a localized distribution. Below we are going to show this behavior rigorously for the action \eqref{eq:action}.

\begin{figure}[!ht]
\begin{overpic}[width=8cm]{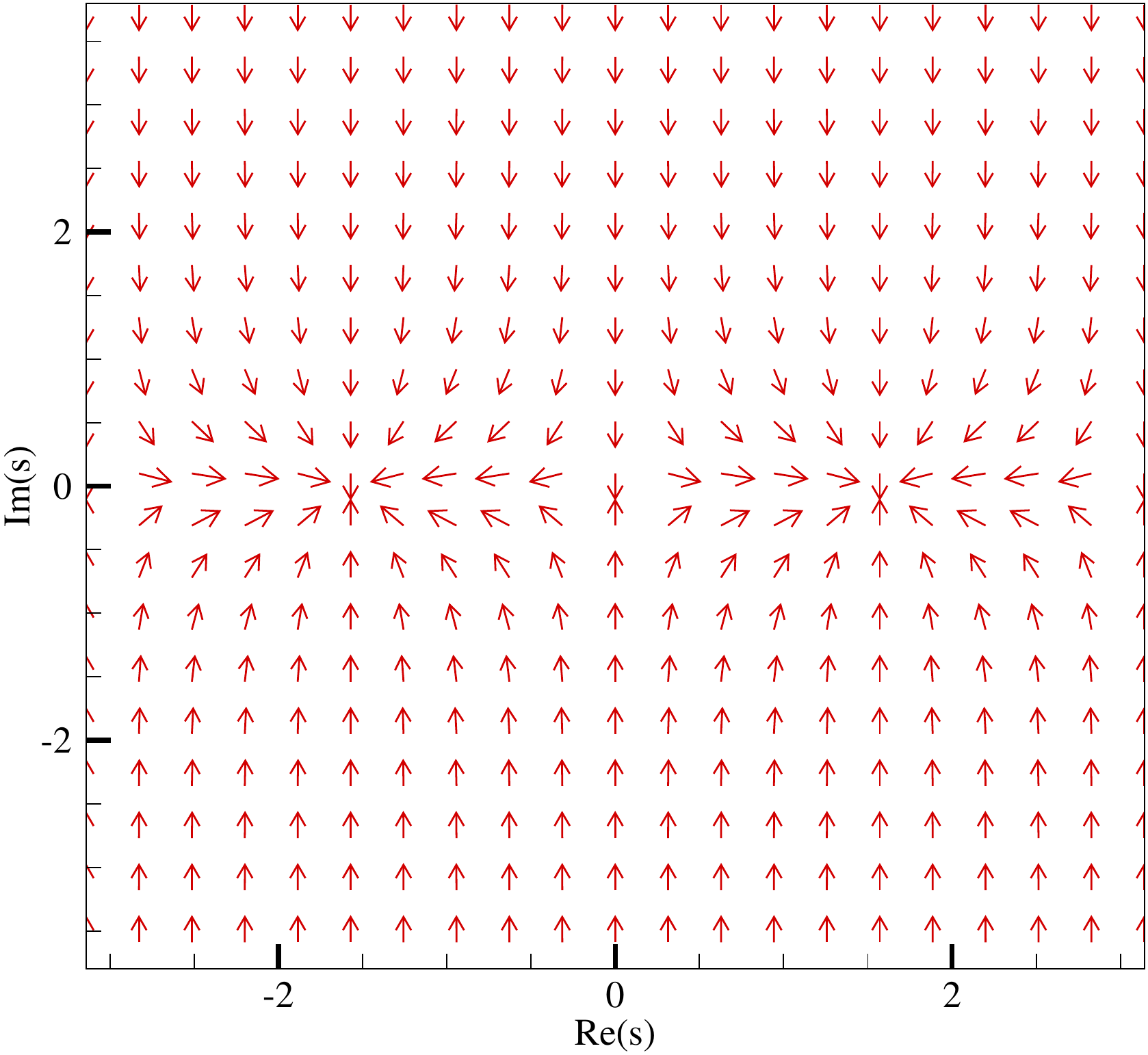}
\end{overpic}
\caption{The flow field $v(s) = 2\cot s$} \label{Fig:flow_k0}
\end{figure}

From the action $S(U)$ given by \eqref{eq:action} and the definition of the Lie derivative \eqref{eq:Lie_derivative}, it is easy to get
\begin{equation*}
K_3 = D_3 S(U) = -\mathrm{i} \beta_1 \tr \left[ \diag \left(\mu,\frac{1}{\mu}\right)\lambda_3^{2} \right] + \mathrm{i} \beta_2 \tr \left[\diag\left(\frac{1}{\mu}, \mu\right) \lambda_3^{2} \right]
= \mathrm{i} (\beta_1 + \beta_2) \left( \frac{1}{\mu} - \mu \right).
\end{equation*}
For simplicity, we let $\mathrm{d}w =\mathrm{d} w_3$ and $\beta = \beta_1 + \beta_2$. Noting that $\mu^{-1} - \mu = 2\mathrm{i} \sin s$, we can rewrite the equation \eqref{eq:dmu2} as
\begin{equation} \label{eq:ds_Polyakov}
\mathrm{d}s = 2 (-\beta \sin s + \cot s) \,\mathrm{d}t + \mathrm{d}w.
\end{equation}
To write down the above equation more explicitly, we let $s = x + \mathrm{i}y$ and $\beta = A + \mathrm{i} B$, where $x,y,A,B \in \mathbb{R}$. Then $K_3$ can be expanded as
\begin{equation}
\begin{split}
K_3 &= 2 \left( - A \cosh y \sin x + B \sinh y \cos x + \frac{\sin 2x}{\cosh 2y - \cos 2x} \right) \\
& \quad - 2 \mathrm{i} \left( A \sinh y \cos x + B \cosh y \sin x + \frac{\sinh 2y}{\cosh 2y - \cos 2x} \right).
\end{split}
\end{equation}
Below we use $K_R$ and $K_I$ to denote the real and imaginary parts of $K_3$. Then the SDE \eqref{eq:ds_Polyakov} is equivalent to the following 2D SDEs:
\begin{equation} \label{eq:ds}
\left\{
\begin{aligned}
\mathrm{d}x &= K_R \,\mathrm{d}t + \mathrm{d}w = 2 \left( - A \cosh y \sin x +B \sinh y \cos x + \frac{\sin 2x}{\cosh 2y - \cos 2x} \right) \mathrm{d}t + \mathrm{d}w,\\
\mathrm{d}y &= K_I \,\mathrm{d}t = - 2 \left( A \sinh y \cos x + B \cosh y \sin x + \frac{\sinh 2y}{\cosh 2y - \cos 2x} \right) \mathrm{d}t,
\end{aligned}
\right.
\end{equation}
where we remind the readers that $x$ is a $2\pi$-periodic variable and $w$ is the real Brownian motion.

Below we would like to claim that when $S(U)$ is small, in the complex Langevin process \eqref{eq:ds}, the imaginary part of $s$ cannot excurse to infinity due to the cotangent term in \eqref{eq:ds2} if the initial $y$ is zero. This is detailed in the following theorem: 
 
\begin{theorem} \label{Thm:beta}
Suppose $|A| < \frac{3\sqrt{3}}{2}$. There exists $\epsilon_A > 0$ such that when $B \in (-\epsilon_A, \epsilon_A)$, we can find constants $C_1$ and $C_2$ satisfying all the following conditions:
\begin{itemize}
\item $C_2 > C_1 \geqslant 0$;
\item $K_I < 0$ for all $x \in \mathbb{R}$ and $y \in (C_1, C_2)$;
\item $K_I > 0$ for all $x \in \mathbb{R}$ and $y \in (-C_2, -C_1)$.
\end{itemize}
Here $K_I$ is given in the second equation of \eqref{eq:ds}.
\end{theorem}

\begin{proof}
By the definition of $K_I$, we see that it depends on $x,y,A,B$, and it satisfies the following identities:
\begin{equation*}
K_I(-x,-y;A,B) = K_I(x+\pi,-y;-A,B) = K_I(\pi-x,y;A,-B)= -K_I(x,y;A,B).
\end{equation*}
Such a symmetry allows us to focus only on the case with $y<0$, $A \geqslant0$ and $B>0$. Let $\xi = \cos x$ and $\eta = -\tanh y$ which satisfy $-1 \leqslant \xi \leqslant 1$ and $0<\eta<1$. Using these new variables, we can rewrite the inequality $K_I > 0$ in the following equivalent form:
\begin{equation} \label{eq:ineqKI}
\sqrt{\frac{1}{1-\eta^2}} (-A\eta \xi + B) < \frac{\eta}{1-\xi^2(1-\eta^2)}.
\end{equation}
Below we consider the cases $A=0$ and $A>0$ separately.

If $A=0$, in order that \eqref{eq:ineqKI} holds for all $x \in \mathbb{R}$, we need to ensure that
\begin{equation} \label{eq:A=0}
B \sqrt{\frac{1}{1-\eta^2}} < \min_{\xi \in [-1,1]} \frac{\eta}{1-\xi^2(1-\eta^2)} = \eta.
\end{equation}
Clearly, the above inequality holds for $B = 0$ and any $\eta = 1/2$. Thus, by the continuity of both sides of \eqref{eq:A=0}, we can find a neighborhood of $(0,1/2)$, say $(-\epsilon_0, \epsilon_0) \times (1/2-\delta, 1/2+\delta)$, such that \eqref{eq:A=0} still holds, which indicates $K_I > 0$. Thus, by the relation between $y$ and $\eta$, the constants $C_1$ and $C_2$ can be found by $C_1 = \arctanh (1/2-\delta)$ and $C_2 = \arctanh (1/2+\delta)$.

If $A>0$, the equation \eqref{eq:ineqKI} can be written as
\begin{equation} \label{eq:ABineq}
\xi^3(1-\eta^2)-\xi - \frac{B}{A\eta}\xi^2(1-\eta^2) + \frac{B}{A\eta} < \frac{1}{A} \sqrt{1-\eta^2}.
\end{equation}
Similar to the previous case, we just need to find some $\eta > 0$ such that the above inequality holds for $B=0$ and any $\xi \in [-1,1]$. This condition is equivalent to
\begin{equation} \label{eq:A>0}
\max_{\xi \in [-1,1]} [\xi^3(1-\eta^2)-\xi] < \frac{1}{A} \sqrt{1-\eta^2}.
\end{equation}
When $0 < \eta < 1$, it is not difficult to work out that $\max\limits_{\xi \in [-1,1]} [\xi^3(1-\eta^2)-\xi] = \frac{2}{3}\sqrt{\frac{1}{3(1-\eta^2)}}$. Thus a simple rearrangement of \eqref{eq:A>0} yields
\begin{equation*}
\eta^2 < 1 - \frac{2A}{3\sqrt{3}}.
\end{equation*}
By the assumption $0 < A < 3\sqrt{3}/2$, we know that the right-hand side of the above equation is positive, and therefore a positive $\eta$ satisfying \eqref{eq:A>0} exists. The existence of $\epsilon_A$, $C_1$ and $C_2$ follows the same argument as in the case $A = 0$.
\end{proof}

Theorem \ref{Thm:beta} shows that when $|A| < 3\sqrt{3}/2$ and $B \in (-\epsilon_A, \epsilon_A)$, the solution of \eqref{eq:ds} satisfies $y(\eta) \in (-C_2, C_2)$ if the initial condition satisfies $y(0) \in (-C_2, C_2)$. Recall that the formal justification of the complex Langevin method requires that $y(0) = 0 \in (-C_2, C_2)$. Therefore Theorem \ref{Thm:beta} does provide us a safe region of parameters, for which the imaginary part cannot excurse out of the striplike area $|y| < C_2$. A more precise description of this safe region can be obtained by \eqref{eq:ABineq}, which inspires us to define
\begin{displaymath}
f(A,B) = \inf_{\eta \in (0,1)} \max_{\xi \in [-1,1]}
  \left( \xi^3(1-\eta^2)-\xi - \frac{B}{A\eta}\xi^2(1-\eta^2) + \frac{B}{A\eta} - \frac{1}{A} \sqrt{1-\eta^2} \right),
\end{displaymath}
so that for any given $A \in (-3\sqrt{3}/2, 3\sqrt{3}/2)$, the maximum value of $\epsilon_A$ satisfies $f(A, \epsilon_A) = 0$. By this argument, the range of $A$ and $B$ which localizes the distribution function can be numerically computed, and we plot such a region in Figure \ref{Fig:AB}.
\begin{figure}[!ht]
\begin{overpic}[width=8cm]{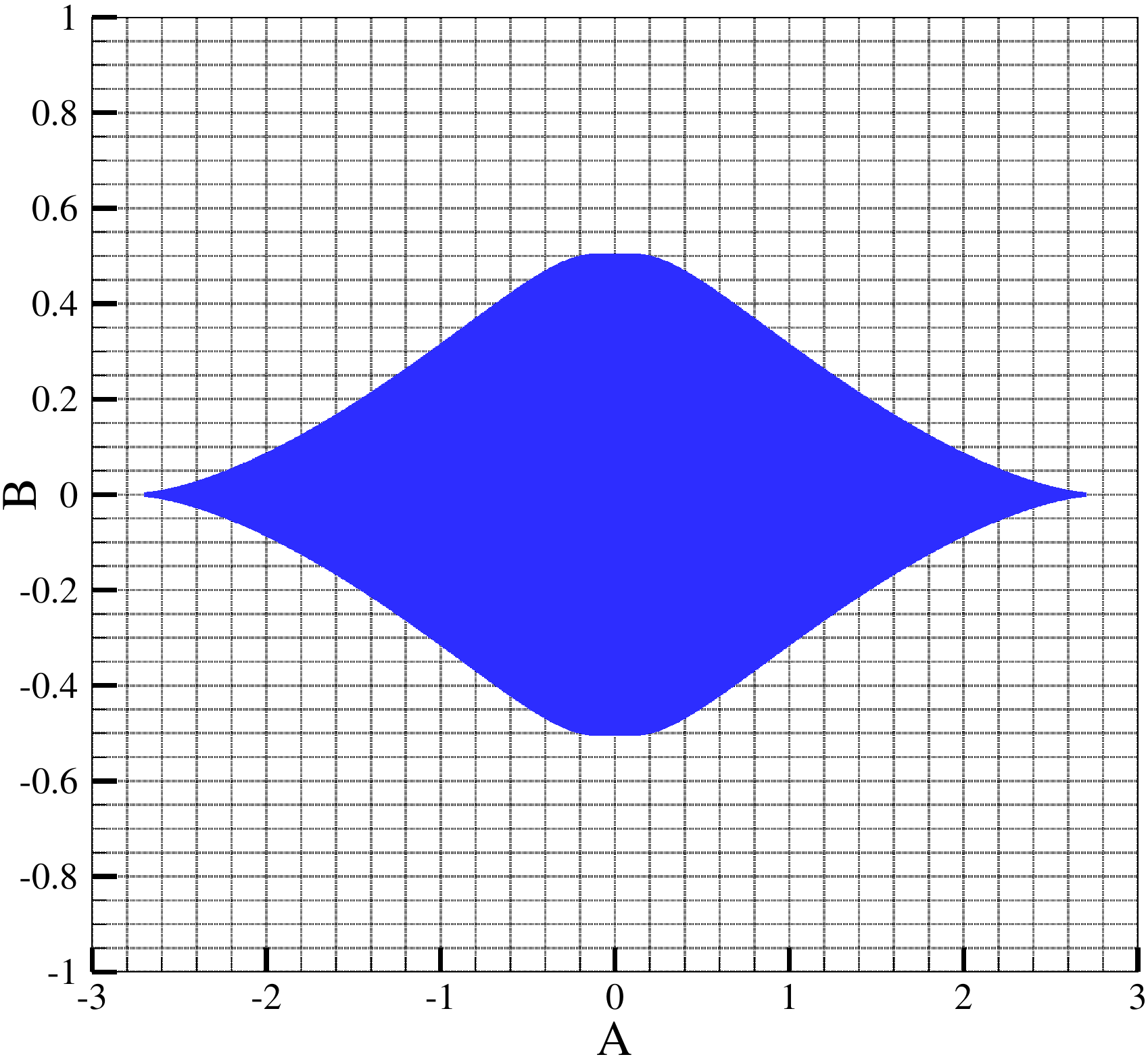}
\end{overpic}
\caption{The region for $(A,B)$ in which the distribution function is localized.} \label{Fig:AB}
\end{figure}
Such a localizing phenomenon has been studied in \cite{Aarts2013}, where the striplike region is the result of the competition of two parts of the action: one part of the action attracts the Langevin process toward the real axis, while the other part repulses the process away from the real axis. In our work, we have shown that gauge cooling provides an additional attractive drift, which helps the formation of the confining strip even if the action does not contain an attractive part. 

One side effect of this advantage of gauge cooling is the existence of singularity at points $s = k\pi$ for $k \in \mathbb{Z}$, corresponding to the case $\mu_1 = \mu_2$. For instance, the SDE around $s = 0$ behaves asymptotically like 
\begin{displaymath}
\left\{ \begin{aligned}
\mathrm{d}x &= \frac{2x}{x^2+y^2} \,\mathrm{d}t + \mathrm{d}w, \\
\mathrm{d}y &= -\frac{2y}{x^2+y^2} \,\mathrm{d}t.
\end{aligned} \right.
\end{displaymath}
Such a velocity field does not even fit in the $L^2$ space, causing that most techniques to prove the ergodicity fail to work. Theoretically, it is yet unclear whether such a singularity will lead to the failure of the complex Langevin method, while our numerical experiments seem to suggest its mildness.

%% file: Examples.tex
\section{Numerical examples} \label{sec:examples}

In this section, we give two examples to validate our numerical analysis. In the first example, we apply different gauge cooling methods to the $SU(3)$ theory to compare their results; in the second example, we simulate \eqref{eq:ds} directly for different parameters to examine the convergence behaviors.

\subsection{Polyakov loop model for $SU(3)$ theory}

In this example, we apply the complex Langevin method with gauge cooling to the Polyakov loop model \eqref{eq:action} for $SU(3)$ theory. The aim is to show the effectiveness of optimal gauge cooling in reducing the norm $\|\{U\}\|$. Following \cite{Seiler2013}, we choose $\beta_1=\beta+\kappa e^\mu$ and $\beta_2 = \bar{\beta}+\kappa e^{-\mu}$ in \eqref{eq:action}, and set $\beta = 2$, $\kappa = 0.1$ and $\mu = 1$. In the following numerical experiments, the following three methods will be tested:
\begin{itemize}
\item Complex Langevin with no gauge cooling.
\item Complex Langevin with optimal gauge cooling given by \eqref{eq:diagonal}.
\item Complex Langevin with gauge cooling implemented by gradient descent method \cite{Seiler2013}.
\end{itemize}
In the last method, we follow \cite{Seiler2013} to assume that $V_{ak}$ are purely imaginary ($V_{ak} = \mathrm{i} Y_{ak}$), and the gradient has been given in the second equation of \eqref{eq:gradient}. Thus every gradient descending step is given by
\begin{equation} \label{eq:gc_gd}
Y_{ak} \leftarrow -2 \tilde{\alpha} \tr[\lambda_a(U_k U_k^{\dagger} - U_{k-1}^{\dagger} U_{k-1})],
  \qquad U_k \leftarrow \exp \left(\sum_{a=1}^8 Y_{ak} \lambda_a \right) U_k \exp \left(-\sum_{a=1}^8 Y_{a,k+1} \lambda_a \right).
\end{equation}
Here $\lambda_a$ is the Gell-Mann matrices given in \eqref{eq:Gell-Mann}, and $\tilde{\alpha}$ indicates the step length. Since gauge cooling is to be done after every time step, and the number of time steps may be large, here we apply neither sophisticated line search technique nor accurate stopping criterion. The step length is simply chosen as $\tilde{\alpha} = \alpha \Delta t$ for some constant $\alpha$. After every time step of complex Langevin, the gradient-descent based gauge cooling \eqref{eq:gc_gd} is applied five times. Two choices of $\alpha$ ($0.4$ and $1.0$) are tested.

In our numerical tests, we set the Langevin time step to be $\Delta t = 2 \times 10^{-5}$, and we apply the above methods until $t = 10$. The evolution of the norm $\|\{U\}\|$ for both cases $N=16$ and $N=32$ is provided in Figure \ref{Fig:gs}, where the horizontal axis is the Langevin time, and the vertical axis represents the value of $\Delta F := \|\{U\}\|^2 - N n$, which is zero if every $U_k$ is unitary. Figure \ref{Fig:gs} only provides one realization of the complex Langevin process in each case. In our experiments, despite the existence of the randomness, no essential difference can be observed in different runs. Figure \ref{Fig:gs} shows that if no gauge cooling is used, the numerical method fails due to a rapidly increasing $\Delta F$, and all other three lines show converging results. For $N=16$, the gradient descent methods show a quite good performance. The value of $\Delta F$ are well controlled below $10^{-2}$ for all time steps, and two different step lengths $\alpha = 0.4$ and $\alpha = 1.0$ show very similar results. However, when $N$ increases to $32$, the step length given by $\alpha = 1.0$ shows much better performance. When $t > 7$, even for $\alpha = 1.0$, the value of $\Delta F$ can reach a number larger than $1$. Contrarily, if the optimal gauge cooling is used, the value of $\Delta F$ is just oscillating around $10^{-6}$, which confirms our statement that the process is essentially independent of $N$ if the optimal gauge cooling is used. These facts indicate that when the number of links is large, one may need to design a more efficient optimization method to suppress the norm $\|\{U\}\|$, so that the power of gauge cooling can be freed maximally.

\begin{figure}[!ht]
\includegraphics[width=8cm]{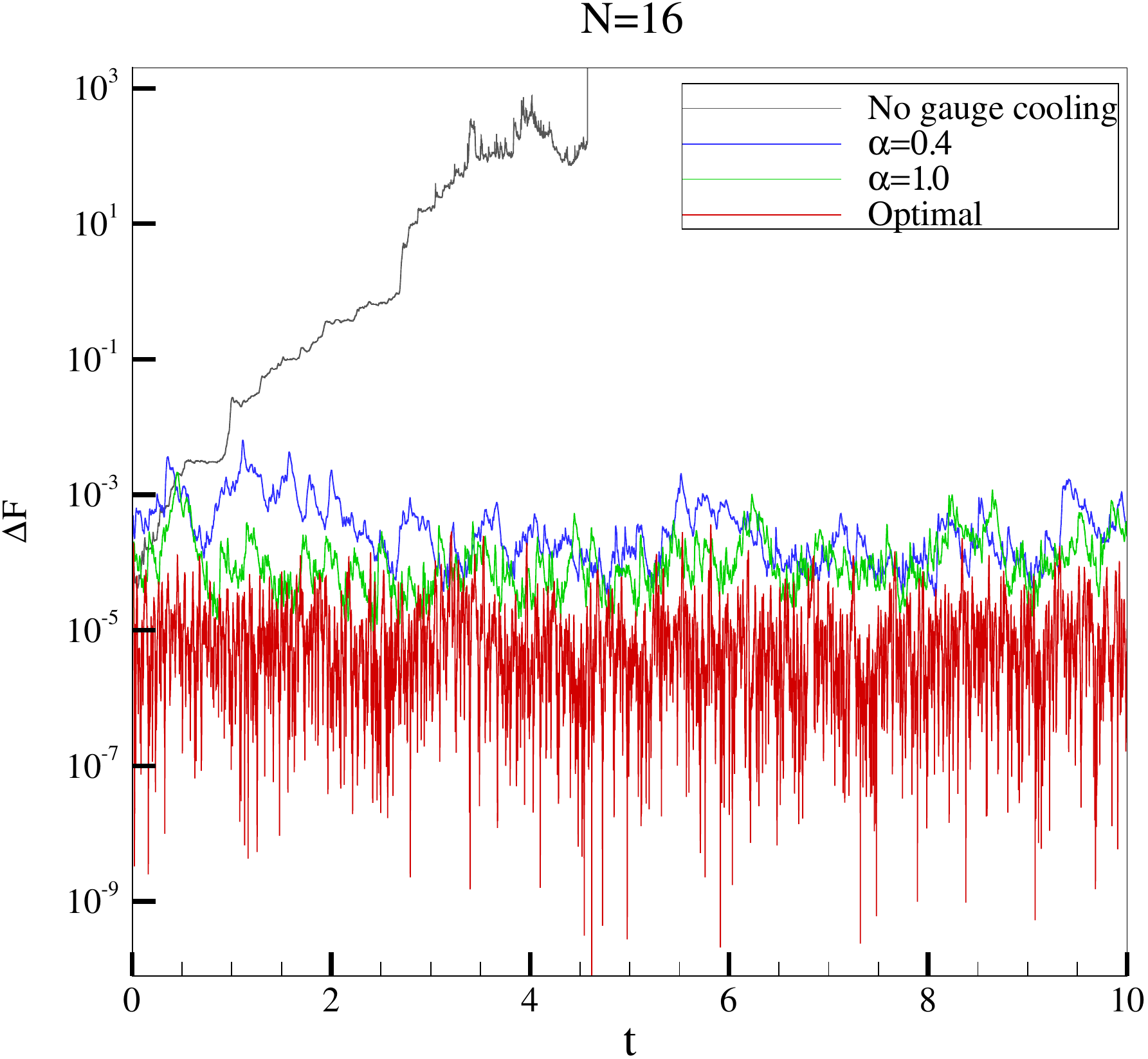}  \quad
\includegraphics[width=8cm]{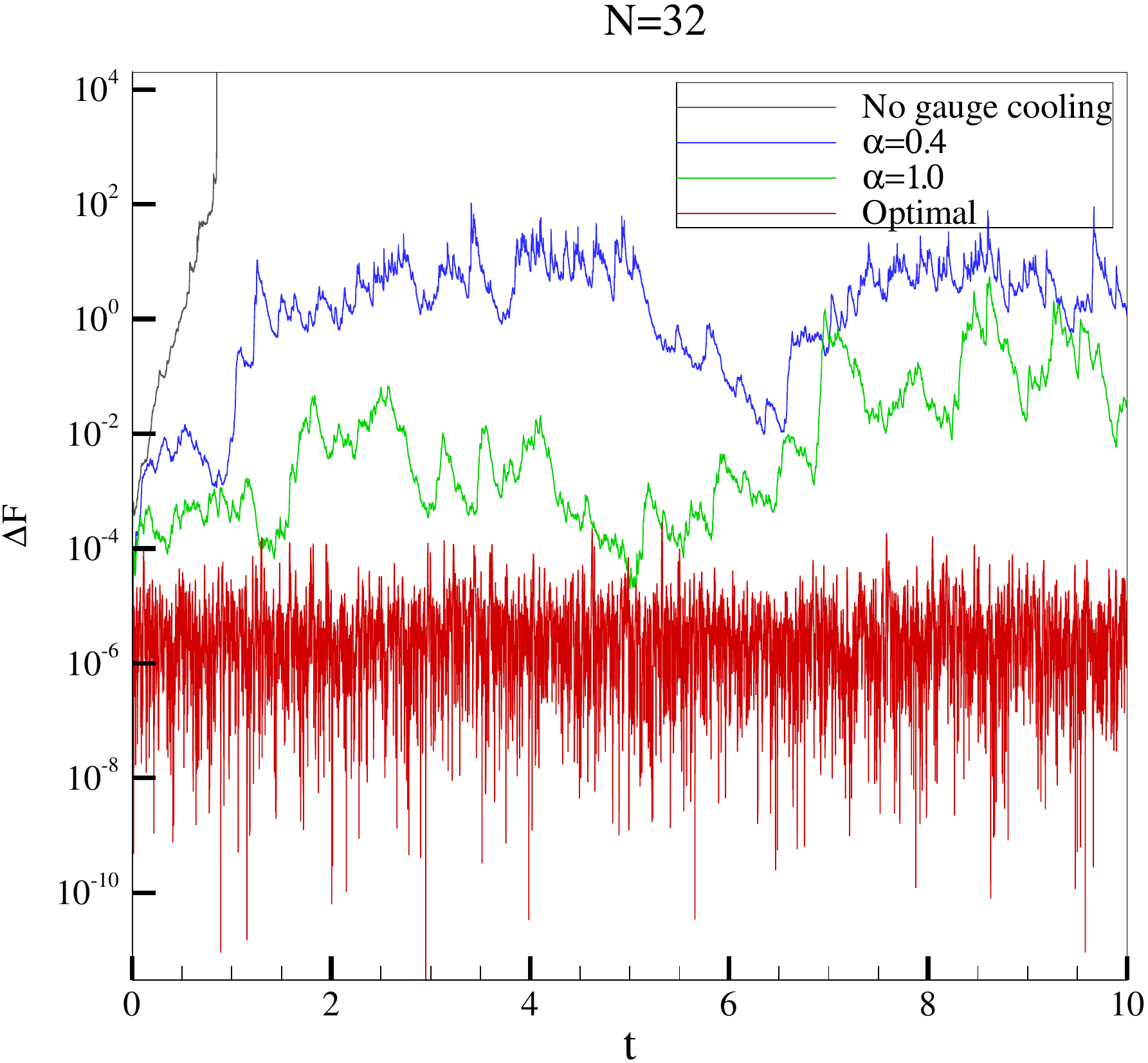}
\caption{The evolution of $\Delta F$ with different gauge cooling techniques}   \label{Fig:gs}
\end{figure}

To verify the correctness of the numerical method, we have calculated the observables $O_k(\{U\}) = \tr\left[(U_1 \cdots U_N)^k\right]$ for $k=\pm1,\pm2,\pm3$. Because the Polyakov loop model reduces to a one-link integral analytically \cite{Seiler2013}, we can calculate the exact values of these observables by considering the case $N=1$. For any $U \in SU(3)$, its eigenvalues can be written as $e^{\mathrm{i}\phi_1},e^{\mathrm{i}\phi_2}$ and $e^{-\mathrm{i}(\phi_1+\phi_2)}$ where $\phi_1,\phi_2 \in (-\pi,\pi]$. By Weyl's integral theorem \cite{Hall2015}, when $N=1$, we can calculate the integrals $\langle O_k \rangle$ as
\begin{equation*}
\langle O_k \rangle = \frac{1}{\eta Z} \int_{-\pi}^\pi \int_{-\pi}^\pi 
(e^{\mathrm{i}k\phi_1}+e^{\mathrm{i}k\phi_2}+e^{-\mathrm{i}k(\phi_1+\phi_2)}) 
\exp (-\widetilde{S}(\phi_1,\phi_2))  |V(\phi_1,\phi_2)|^2
\mathrm{d}\phi_1 \mathrm{d}\phi_2,
\end{equation*} 
where $\eta$ is a normalization constant given by
\begin{equation*}
\eta = \int_{-\pi}^\pi \int_{-\pi}^\pi |V(\phi_1,\phi_2)|^2 \mathrm{d}\phi_1 \mathrm{d}\phi_2,
\end{equation*} 
and
\begin{gather*}
\widetilde{S}(\phi_1,\phi_2) = - \beta_1(e^{\mathrm{i}k\phi_1}+e^{\mathrm{i}k\phi_2}+e^{-\mathrm{i}k(\phi_1+\phi_2)}) - \beta_2 (e^{-\mathrm{i}k\phi_1}+e^{-\mathrm{i}k\phi_2}+e^{\mathrm{i}k(\phi_1+\phi_2)}), \\
V(\phi_1,\phi_2) = \det
\begin{pmatrix}
   1&1&1 \\
   e^{\mathrm{i}k\phi_1}&e^{\mathrm{i}k\phi_2}&e^{-\mathrm{i}k(\phi_1+\phi_2)} \\
   e^{2\mathrm{i}k\phi_1}&e^{2\mathrm{i}k\phi_2}&e^{-3\mathrm{i}k(\phi_1+\phi_2)}
\end{pmatrix}
=\left[ e^{\mathrm{i}k\phi_2}-e^{\mathrm{i}k\phi_1}\right]
\left[ e^{-\mathrm{i}k(\phi_1+\phi_2)}-e^{\mathrm{i}k\phi_1}\right]
\left[ e^{-\mathrm{i}k(\phi_1+\phi_2)}-e^{\mathrm{i}k\phi_2}\right], \\
Z = \frac{1}{\eta} \int_{-\pi}^\pi \int_{-\pi}^\pi 
\exp (-\widetilde{S}(\phi_1,\phi_2))  |V(\phi_1,\phi_2)|^2
\mathrm{d}\phi_1 \mathrm{d}\phi_2.
\end{gather*} 
Thus, the value of $\langle O_k \rangle$ can be obtained by two-dimensional numerical integration, which is used to verify the results from stochastic simulations. In Table \ref{tab:observable}, some results have been listed. Because the value of $\langle O_k \rangle$ is real, we just list the real part of the numerical result. Note that the results for CLM without gauge cooling is not listed, since the computation always fails due to the rapidly growing norm of $\|\{U\}\|$. For $N=16$, all the three methods give reasonable approximations of the exact integral value, while for $N=32$, we observe large deviations from the exact values when the gauge cooling is insufficient ($\alpha = 0.4$). In fact, we even experienced divergence in our test runs when $\alpha = 0.4$, which again emphasizes the importance to use an effective gauge cooling method.

\begin{table}[!ht]
\centering
\caption{Numerical results for the observable $\langle O_k \rangle$} \label{tab:observable}
\begin{tabular}{cccccccccc}
\hline 
& & \hspace{10pt} & \multicolumn{3}{c}{$N=16$} & \hspace{10pt} & \multicolumn{3}{c}{$N=32$} \\
\hline
$k$ & Exact & & $\alpha=0.4$ & $\alpha=1.0$ & Optimal & &
  $\alpha=0.4$ & $\alpha=1.0$ & Optimal \\
\hline
$1$ & $2.0957$ & & $2.1001$ & $2.1042$ & $2.1017$ & &
$2.1394$ & $2.0900$ & $2.0966$ \\
$-1$ & $2.1026$ & & $2.1070$ & $2.1109$ & $2.1083$ & &
$2.1422$  & $2.0958$ & $2.1033$ \\
$2$ & $0.3761$ & & $0.3850$ & $0.3894$ & $0.3775$ & &
$0.6548$  & $0.3583$ & $0.3750$ \\
$-2$ & $0.4092$ & & $0.4182$ & $0.4224$ & $0.4100$ & &
$0.5463$  & $0.3874$ & $0.4078$ \\
$3$ & $-0.5269$ & & $-0.5270$ & $-0.5272$ & $-0.5482$ & &
$-9.0467$  & $-0.5514$ & $-0.5247$ \\
$-3$ & $-0.4800$ & & $-0.4803$ & $-0.4802$ & $-0.5014$ & &
$-2.2547$ & $-0.4938$ & $-0.4778$ \\
\hline
\end{tabular} 
\end{table}

\subsection{Optimal gauge cooling in the $SU(2)$ theory} \label{sec:ex2}
In this example, we focus on the behavior of equation \eqref{eq:ds} for different values of $A$ and $B$. For all simulations, we start with the initial value $x=0.5$ and $y=0$, and evolve the process with time step $\Delta t = 10^{-5}$. The first $3 \times 10^5$ steps are considered as ``unsteady calculations'', and afterwards, we proceed the evolution by $10^8$ steps, in which the samples are drawn at every $10^4$ steps. 

The observables we compute are still $O_k(U) = \tr(U^k)$, where $U = \diag(\mathrm{e}^{-\mathrm{i}s},\mathrm{e}^{\mathrm{i}s})$. The results for $k = 1,2,3$ are provided in Table \ref{tab:ex2}, and we also plot the distributions of the samples $(x,y)$ in Figure \ref{Fig:eig}. The following four choices of $A$ and $B$ are considered in our results:
\begin{itemize}
\item $A = 1$ and $B = 0.2$: In this case, the conclusion of Theorem \ref{Thm:beta} holds, and therefore all the samples are confined in the strip-like region. The simulation is stable and provides reliable values of the observables.
\item $A = 1$ and $B = 2$: The conclusion of Theorem \ref{Thm:beta} does not hold in this case, so the possible values of $y$ can be arbitrary large. The numerical values of the observables also show an interesting behavior. When $k$ increases, the numerical result shows larger and larger deviation from the exact value. In fact, different simulations generate very different results for $k=2$ and $k=3$, while the results for $k=1$ seem stable in different runs. This behavior may be due to insufficient decay in the invariant measure, which is studied in detail in \cite{Scherzer2019}.
\item $A = 5$ and $B = 1$: Since $A > 3\sqrt{3}/2$, Theorem \ref{Thm:beta} does not hold. However, the value of $B$ is relatively small compared with $A$, which provides only a small channel to allow $y$ to drift far from the real axis. Therefore, in our numerical results, the distribution of the samples still looks as if they are confined. The simulation is again stable and the numerical results look sound.
\item $A = 5$ and $B = 10$: This looks like a severe case since we have both large $A$ and large $B$. However, the samples are distributed well around the origin, showing a clear existence of the invariance measure. It shows that the complex Langevin method may work well beyond our theory, which also requires future works.
\end{itemize}

\begin{table}[!ht]
\centering
\caption{Numerical results for the observable $\langle O_k \rangle$}
\label{tab:ex2}
\begin{tabular}{ccccccc}
\hline 
 & \hspace{10pt} & \multicolumn{2}{c}{$A=1,B=0.2$} & \hspace{10pt} & \multicolumn{2}{c}{$A=1,B=2$} \\
\hline
$k$ & & Exact & Numerical & & Exact & Numerical \\
\hline
$1$ & & $0.8759 + 0.1300\mathrm{i}$ & $0.8569+0.1332\mathrm{i}$ & & $1.7449 + 0.5495\mathrm{i}$ & $1.5447+0.5269\mathrm{i}$ \\
$2$ & & $-0.6017 + 0.1304\mathrm{i}$ & $-0.6304+0.1298\mathrm{i}$ & & $0.2936 + 1.7642\mathrm{i}$ & $-0.1957+0.9220\mathrm{i}$ \\
$3$ & & $-0.7562 - 0.0652\mathrm{i}$ & $-0.7543-0.0738\mathrm{i}$ & & $-2.1127 + 1.2080\mathrm{i}$ & $20.6438-34.4723\mathrm{i}$ \\
\hline
\hline
& \hspace{10pt} & \multicolumn{2}{c}{$A=5,B=1$} & \hspace{10pt} & \multicolumn{2}{c}{$A=5,B=10$} \\
\hline
$k$ & & Exact & Numerical & & Exact & Numerical \\
\hline
$1$ & & $1.7189 + 0.0544\mathrm{i}$ & $1.7187+0.0545\mathrm{i}$ & & $1.9390 + 0.1188\mathrm{i}$ & $1.9395+0.1195\mathrm{i}$ \\
$2$ & & $1.0021 + 0.1669\mathrm{i}$ & $1.0013+0.1674\mathrm{i}$ & & $1.7388 + 0.4511\mathrm{i}$ & $1.7406+0.4543\mathrm{i}$ \\
$3$ & & $0.1531 + 0.2341\mathrm{i}$ & $0.1513+0.2345\mathrm{i}$ & & $1.3565 + 0.9231\mathrm{i}$ & $1.3601+0.9318\mathrm{i}$ \\
\hline
\end{tabular} 
\end{table}

\begin{figure}[!ht]
\includegraphics[width=7.5cm,clip]{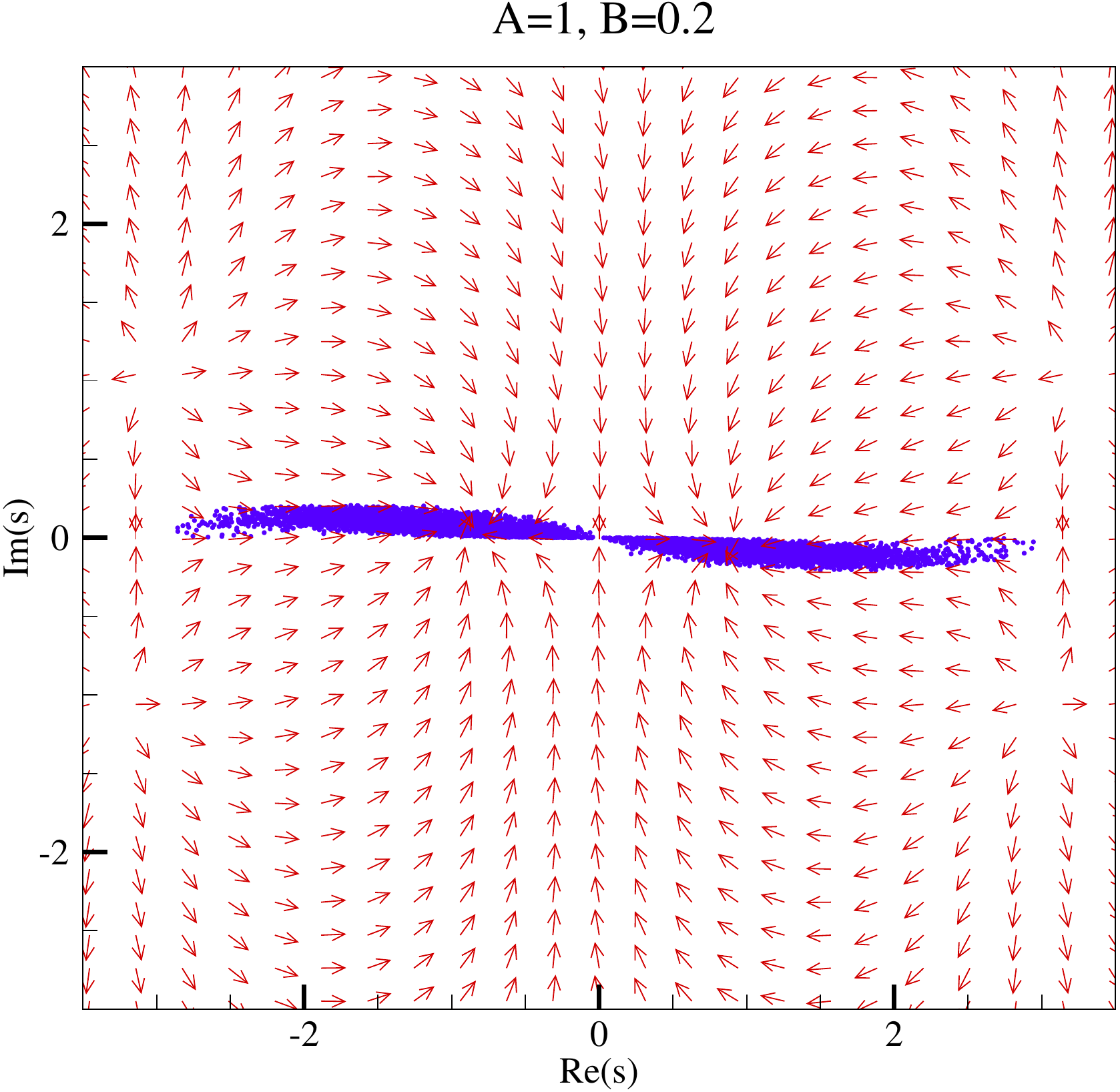} \quad 
\includegraphics[width=7.5cm,clip]{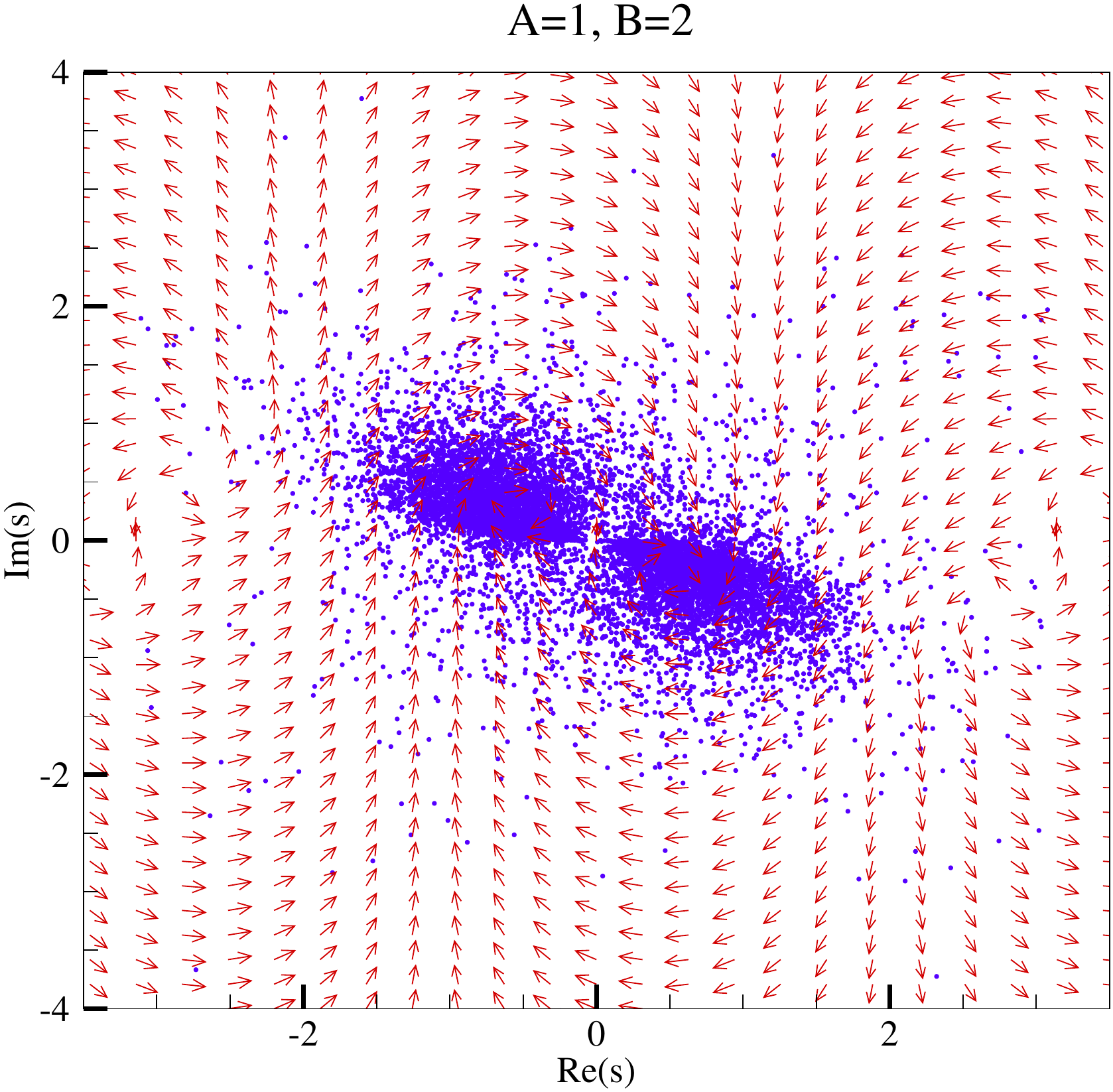} \\[7pt]
\includegraphics[width=7.5cm,clip]{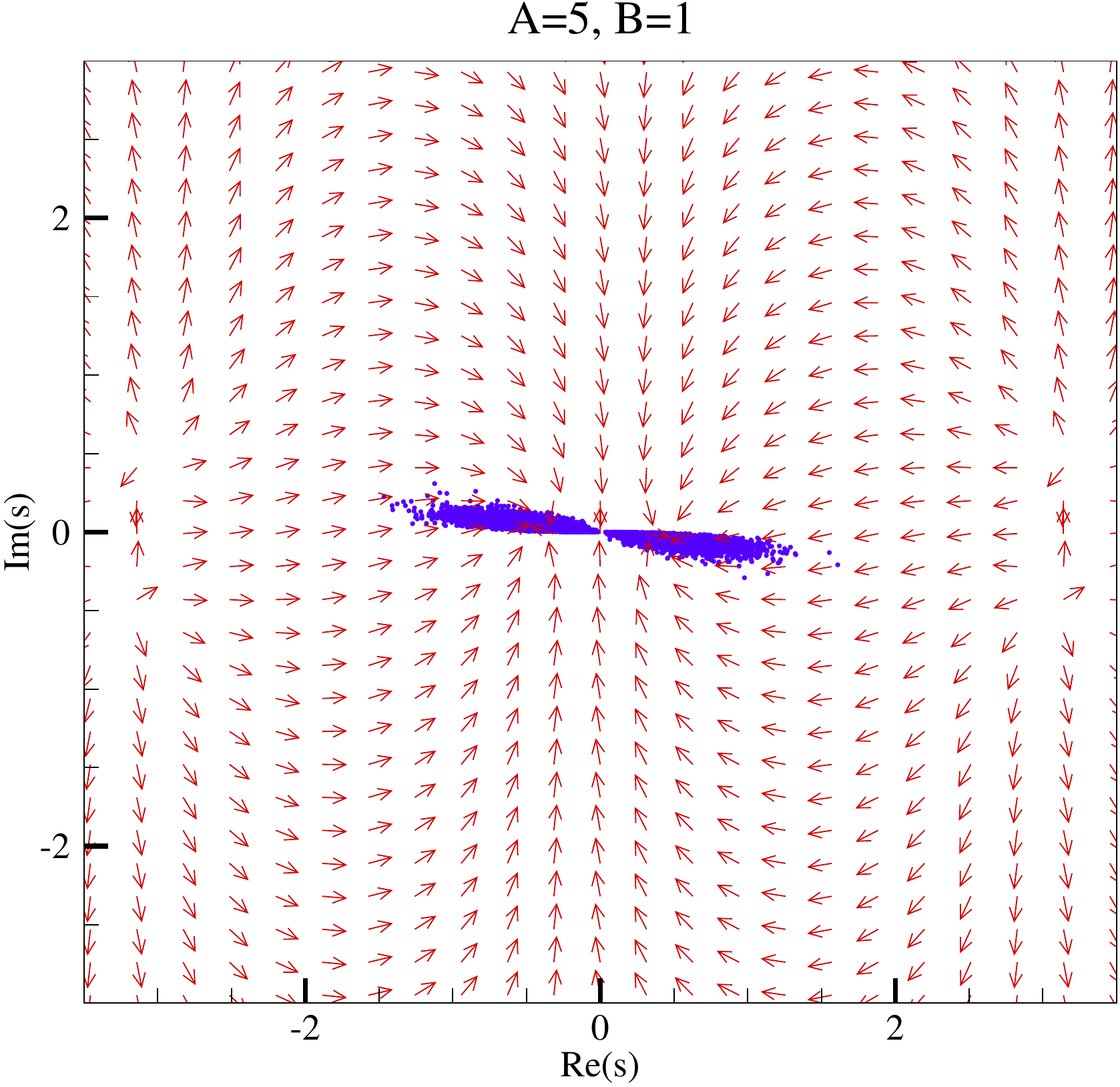} \quad 
\includegraphics[width=7.5cm,clip]{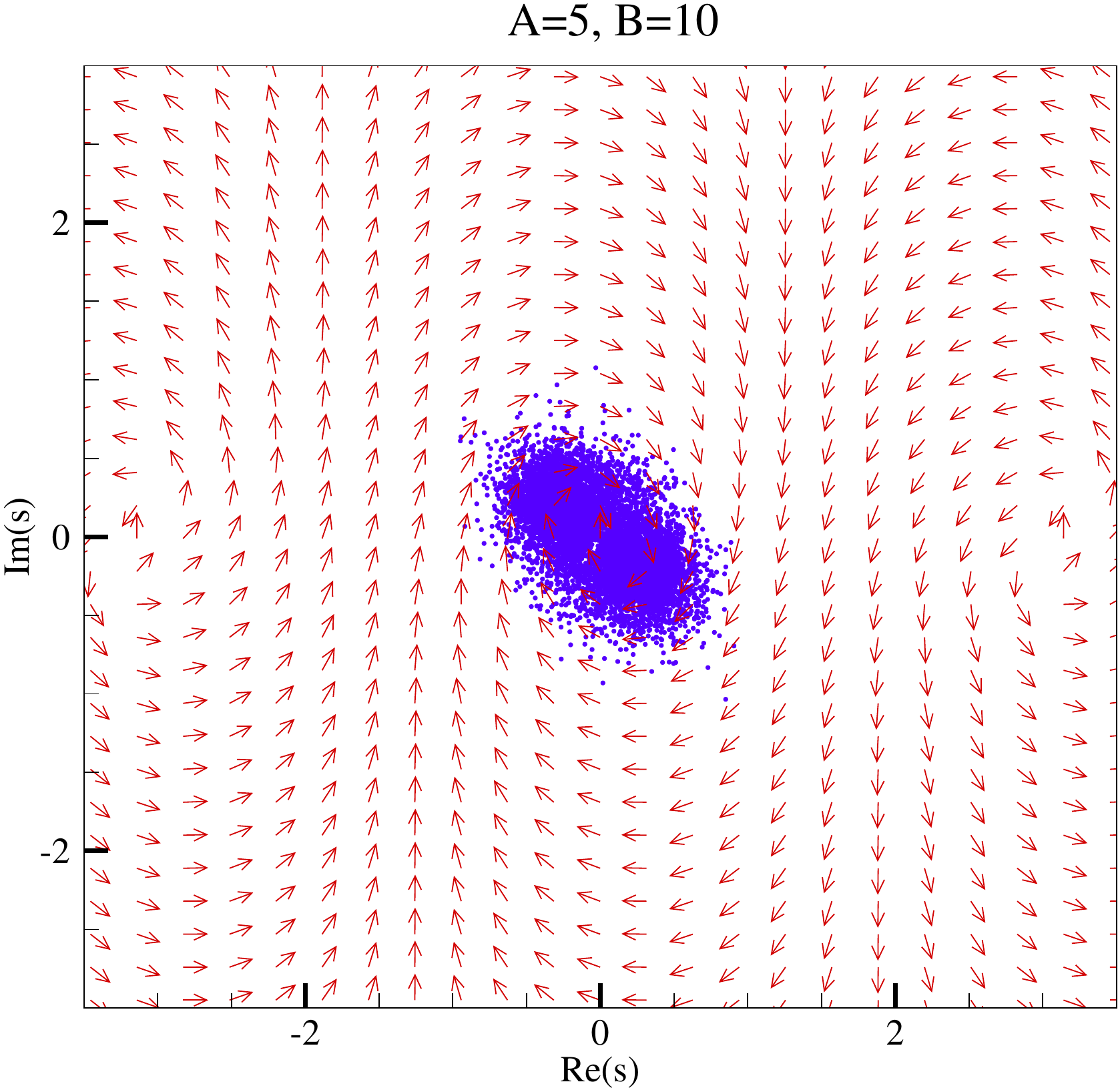} \\
\caption{The distribution of samples. The red arrows are the field of drift velocity, which is normalized for better visualization.} \label{Fig:eig}
\end{figure}

Another observation worth mentioning is that the singularities introduced by gauge cooling seem not to have a bad effect on the numerical results. The failure in the case $A=1$ and $B=2$ looks more like the result of slow decaying distribution function, since $\mathrm{e}^{\mathrm{i}ks}$ increases faster in the imaginary direction when $k$ is larger, which may yield worse results for the observable $\langle O_k \rangle$. However, in the cases $A=1, B=0.2$ and $A=5, B=1$ as shown in Figure \ref{Fig:eig}, we do observe nonsmoothness of distribution functions at the origin, which is clearly the effect of the singular drift.

%% file: Conclusion.tex
\section{Conclusion} \label{sec:conclusion}
Due to the accessibility of the exact solution for optimal gauge cooling problem in the one-dimensional case, we have presented a clear picture for the stabilizing effect of gauge cooling for CLM. As a summary, we would like to point out the following four effects of gauge cooling:
\begin{itemize}
\item A large number of redundant degrees of freedom are removed (see Theorem \ref{thm:opt_cooling});
\item Some components of the drift velocity no longer take effect (see Section \ref{sec:SU(n)});
\item Additional drift velocity toward the non-complexified domain is introduced, causing possible confinement of the samples (see Section \ref{sec:SU(2)});
\item Singularities are introduced to the drift velocity (see Section \ref{sec:SU(2)}).
\end{itemize}
The first three effects clearly help stabilize the CLM, while the outcome of the last one remains unclear. Our analysis clearly emphasizes the importance of a good gauge cooling scheme in the multidimensional case, which is part of our ongoing work. Besides, our numerical results in Section \ref{sec:ex2} suggest that the applicability of CLM with gauge cooling should be far beyond the theory given by Theorem \ref{Thm:beta}, as will also be further studied in our future work.